\documentclass[11pt]{amsart}
\usepackage{fullpage,url,amssymb,enumerate,colonequals}

\usepackage{mathrsfs} 
\usepackage[justification=centering]{caption}
\usepackage{MnSymbol}
\usepackage{extarrows}
\usepackage{lscape}
\usepackage[all,cmtip]{xy}

\makeatletter
\renewcommand\part{%
   \if@noskipsec \leavevmode \fi
   \par
   \addvspace{4ex}%
   \@afterindentfalse
   \secdef\@part\@spart}

\def\@part[#1]#2{%
    \ifnum \c@secnumdepth >\m@ne
      \refstepcounter{part}%
      \addcontentsline{toc}{part}{\thepart\hspace{1em}#1}%
    \else
      \addcontentsline{toc}{part}{#1}%
    \fi
    {\parindent \z@ \raggedright
     \interlinepenalty \@M
     \normalfont
     \ifnum \c@secnumdepth >\m@ne
       \large\bfseries \partname\nobreakspace\thepart.
     \fi
     \large \bfseries #2%
     \par}%
    \nobreak
    \vskip 3ex
    \@afterheading}
\def\@spart#1{%
    {\parindent \z@ \raggedright
     \interlinepenalty \@M
     \normalfont
     \large \bfseries #1\par}%
     \nobreak
     \vskip 3ex
     \@afterheading}
\makeatother

\usepackage{color}

\usepackage{xr-hyper}

\usepackage[
        colorlinks, citecolor=darkgreen,
        backref,
        pdfauthor={Nicolas Billerey, Imin Chen, Luis Dieulefait, Nuno Freitas}, 
]{hyperref}
\usepackage{cleveref}
\usepackage{comment}
\usepackage{todonotes}

\newcommand{\C}{\mathbb{C}}
\newcommand{\F}{\mathbb{F}}

\newcommand{\Q}{\mathbb{Q}}

\newcommand{\Z}{\mathbb{Z}}
\newcommand{\Qbar}{{\overline{\Q}}}

\newcommand{\rhobar}{{\overline{\rho}}}

\newcommand{\eps}{\varepsilon}

\newcommand{\calO}{\mathcal{O}}

\newcommand{\Fp}{\mathfrak{p}}
\newcommand{\Fq}{\mathfrak{q}}
\newcommand{\Ff}{\mathfrak{f}}

\DeclareMathOperator{\Aut}{Aut}

\DeclareMathOperator{\End}{End}

\DeclareMathOperator{\Frob}{Frob}
\DeclareMathOperator{\Gal}{Gal}

\DeclareMathOperator{\Norm}{Norm}

\DeclareMathOperator{\tr}{tr}

\newcommand{\unr}{{\operatorname{unr}}}

\newcommand{\vv}{\upsilon}

\newcommand{\GL}{\operatorname{GL}}

\newcommand{\SL}{\operatorname{SL}}

\numberwithin{equation}{section}

\newtheorem{theorem}[equation]{Theorem}
\newtheorem{lemma}[equation]{Lemma}
\newtheorem{corollary}[equation]{Corollary}
\newtheorem{proposition}[equation]{Proposition}

\theoremstyle{definition}

\theoremstyle{remark}
\newtheorem{remark}[equation]{Remark}

\definecolor{darkgreen}{rgb}{0,0.5,0}

\usepackage{circledsteps}
\usepackage{tabularx}
\usepackage{tikz}
\usepackage{amsmath} 

\newcommand\Cc[1]{\tikz[remember picture]{\node(#1)[inner sep=0pt]{\sffamily#1};}}

\makeatletter
\let\@wraptoccontribs\wraptoccontribs
\makeatother

\begin{document}

\title[On Darmon's program, II]{On Darmon's program \\ for the Generalized Fermat equation, II}

\author{Nicolas Billerey}
\address{Laboratoire de Math\'ematiques Blaise Pascal,
    Universit\'e Clermont Auvergne et CNRS, 
    Campus universitaire des C\'ezeaux,
    3, place Vasarely,
    63178 Aubi\`ere Cedex, France}
\email{nicolas.billerey@uca.fr}

\author{Imin Chen}

\address{Department of Mathematics, Simon Fraser University\\
Burnaby, BC V5A 1S6, Canada } \email{ichen@sfu.ca}

\author{Luis Dieulefait}

\address{Departament de Matem\`atiques i Inform\`atica,
Universitat de Barcelona (UB),
Gran Via de les Corts Catalanes 585,
08007 Barcelona, Spain; 
Centre de Recerca Matem\`atica (CRM), Edifici C, Campus Bellaterra, 08193 Bellaterra, Spain}
\email{ldieulefait@ub.edu}

\author{Nuno Freitas}
\address{
Instituto de Ciencias Matem\'aticas, CSIC, 
Calle Nicol\'as Cabrera
13--15, 28049 Madrid, Spain}
\email{nuno.freitas@icmat.es}

\thanks{Billerey was supported by the ANR-23-CE40-0006-01 Gaec project. Chen was supported by NSERC Discovery Grant RGPIN-2017-03892. Freitas was partly supported by the European Union's Horizon 2020 research and innovation programme under the Marie Sk\l{l}odowska-Curie grant 
agreement No.\ 747808 and the grant {\it Proyecto RSME-FBBVA $2015$ Jos\'e Luis Rubio de Francia}. Dieulefait and Freitas were partly supported by the PID2019-107297GB-I00 grant of the MICINN (Spain). Dieulefait was partly supported by the Spanish State Research Agency, through the Severo Ochoa and Mar\'ia de Maeztu Program for Centers and Units of Excellence in R\&D (CEX2020-001084-M)}

\date{\today}


\begin{abstract} 
We obtain additional Diophantine applications of the methods surrounding Darmon's program for the generalized Fermat equation developed in the first part of this series of papers.
As a first application, we use a multi-Frey approach combining two Frey elliptic curves over totally real fields, a Frey hyperelliptic curve over~$\Q$
due to Kraus, and ideas from the Darmon program to give a complete resolution of the generalized Fermat equation $$x^7 + y^7 = 3 z^n$$ for all integers $n \ge 2$. Moreover, we explain how
the use of higher dimensional Frey abelian varieties allows a more efficient proof of this  result due to additional structures that they afford, compared to using only Frey elliptic curves.

As a second application, we use some of these additional structures that Frey abelian varieties possess to show that a full resolution of the generalized Fermat equation
$x^7 + y^7 = z^n$
depends only on the Cartan case of Darmon's big image conjecture.  In the process, we solve the previous equation for solutions $(a,b,c)$ such that~\(a\) and~\(b\) satisfy certain \(2\)- or~\(7\)-adic conditions and all $n \ge 2$. 

\end{abstract}

\maketitle

\section{Introduction}

This is the second part of a series of papers on Darmon's program~\cite{DarmonDuke} for the generalized Fermat equation
\begin{equation}\label{E:rrp}
  x^r + y^r = dz^p,
\end{equation}
with $r$ a fixed prime, $d$ a fixed positive integer and $p$ a prime exponent that is allowed to vary. In the first part~\cite{xhyper_vol1}, building on the progress surrounding the modular method from the last two decades, we analyzed and expanded the limits of this program by developing all the necessary ingredients to use Frey abelian varieties to study~\eqref{E:rrp}. As an application we studied the equation $x^{11} + y^{11} = z^p$.

Here we will explore additional Diophantine consequences of the theory in~\cite{xhyper_vol1} with a focus on computational aspects; these are an essential part of making the program work in practice that were not addressed in the original paper of Darmon.

Let us provide some motivation for the Diophantine equations studied in this paper.
In the works of
Dieulefait--Freitas~\cite{DF2,DF1} and Freitas~\cite{F} several Frey curves are attached to the Fermat equation~\eqref{E:rrp}.
These curves are defined over totally real subfields of the $r$-th cyclotomic field~$\Q(\zeta_r)$. Together with powerful developments of modularity lifting theorems for totally real fields (e.g. \cite{BreuilDiamond}) and consequent modularity of many elliptic curves over these fields, this opened a door to pursue the study of Fermat equations~\eqref{E:rrp}.

Building on these cited works, using a refined multi-Frey curve approach, we have solved equation~\eqref{E:rrp} with $r=5,13$ and $d=3$ for all exponents $p \geq 2$ (see~\cite[Theorems~1 and~2]{BCDF2} and the joint work with Demb\'el\'e~\cite{BCDDF}).  Moreover, the fourth author~\cite[Theorem~1.10]{F} solved it  asymptotically for $r=7$ and $d=3$, i.e., for all primes $p$ sufficiently large. 

Let us recall that a solution $(a,b,c) \in \Z^3$ to~\eqref{E:rrp} is {\it non-trivial} if it satisfies $abc \neq 0$ and we call it {\it primitive} if $\gcd(a,b,c) = 1$. In this paper, we  optimize this latter result by proving the following.

\begin{theorem}\label{T:main}
For all integers $n \geq 2$, there are no non-trivial primitive solutions to
\begin{equation}\label{E:77p}
    x^7 + y^7 = 3 z^n.
\end{equation} 
\end{theorem}

We remark that there is a well known Frey curve~$E_{a,b}/\Q$ attached to~\eqref{E:77p} but it is insufficient to solve the equation (see Section~\ref{S:overQ}).
To highlight the importance of the multi-Frey technique and Darmon's program
both at the theoretical and computational levels,
we will give three proofs of this theorem
which are divided into two types. 

The first type of proof uses the classical modular method with Frey elliptic curves over totally real fields, which is made possible by a previously unobserved and careful selection of  twists of a certain Frey curve. Indeed, our first proof of Theorem~\ref{T:main}, given in Section~\ref{S:Frey7first}, relies on a single Frey curve $F_{a,b}/\Q(\zeta_7)^+$ from~\cite{F} and ideas from~\cite{BCDF2}. Here~\(\Q(\zeta_7)^+\) denotes the (degree-\(3\)) maximal totally real subfield of the \(7\)-th cyclotomic field~\(\Q(\zeta_7)\). A variant of this proof combines both $E_{a,b}$ with~$F_{a,b}$ to reduce the total memory usage, giving a prime example of the advantages of the multi-Frey technique.

The second type of proof adds in the use of Frey abelian varieties following the theory developed in~\cite{xhyper_vol1}. More precisely, we use a multi-Frey approach combining the Frey curves $E_{a,b}$ and~$F_{a,b}$ with Kraus' hyperelliptic curve $C_7(a,b)$ (introduced in \cite[\S 3]{xhyper_vol1}) in two different ways. To illustrate the limits of the theory, the first proof of this type uses the curve $C_7(a,b)$ as much as possible, whilst the second proof is designed to minimize the computational time among all proofs we give.

We find it very interesting and promising that using Frey varieties which, due to their larger field of coefficients should {\it a priori} require harder and lengthier computations, ends up allowing for the most efficient proof. We note that when trying to solve~$\eqref{E:rrp}$ for~$r > 7$ or, more generally, carrying out Darmon's program for other signatures, the spaces of Hilbert newforms involved will quickly become very large. This work and~\cite{xhyper_vol1} show that the Frey varieties have additional structures which one can exploit to reduce computations, despite the fact that we have to work with Jacobians of hyperelliptic curves.

As a second application, we will also prove the following result. Here~\(J(a,b)\) denotes the base change to~\(\Q(\zeta_7)^+\) of the Jacobian of~\(C_7(a,b)\), which is abelian variety of~\(\GL_2\)-type (see~\S\ref{ss:recap} for more details and a summary of the results from~\cite{xhyper_vol1} used in this paper).

\begin{theorem}\label{T:other}
\begin{enumerate}
\item[]
\item\label{T:other_item1} For all integers $n \geq 2$, there are no non-trivial primitive solutions $(a,b,c)$ to the equation
\begin{equation}
\label{main-equ}
    x^{7} + y^{7} = z^n
\end{equation} 
such that $2 \nmid ab$ or $7 \mid a + b$.

\item\label{T:other_item2} Let $p \notin \left\{ 2, 3, 7 \right\}$ be a prime. Then any non-trivial primitive solution $(a,b,c)$ to the equation~\eqref{main-equ} with $n=p$ gives rise to a residual representation $\rhobar_{J(a,b),\Fp}$ such that 
\[
\rhobar_{J(a,b),\Fp} \simeq \rhobar_{J(0,1),\Fp} \otimes \chi
\]
for a character $\chi$ of order dividing $2$ and any choice of prime $\Fp$ of $\Q(\zeta_7)^+$ above~$p$.
\end{enumerate}
\end{theorem}

The coefficient $d = 3$ in~Theorem \ref{T:main} and the $2$- and $7$-adic conditions in part~(\ref{T:other_item1}) of Theorem~\ref{T:other} are due to the Hilbert newforms at the Serre level corresponding to trivial primitive solutions with $ab = 0$, which are well-known obstructions to the success of the modular method. 

It should be remarked that the natural conclusion of part~(\ref{T:other_item2}) in Theorem~\ref{T:other} based on Darmon's original program would also include the possibility that $\rhobar_{J(a,b),\Fp}$ is reducible. A surprising outcome of our study is that we have used Frey elliptic curves whose Galois representations do not arise by specialization of Darmon's Frey representations to `propagate' irreducibility results to Kraus' Frey hyperelliptic curve, also not predicted in the initial program, thereby removing the need for the Borel case in Conjecture 4.1 in \cite{DarmonDuke}.

Conjecture~1.3 in~\cite{xhyper_vol1} is a special case of~\cite[Conjecture 4.1]{DarmonDuke} which is only concerned with the case of normalizer of Cartan subgroups. In the spirit of~\cite[Corollary~1.9]{xhyper_vol1} we prove the following result. 
\begin{corollary}\label{C:reduction2CM}
Assuming Conjecture 1.3 in~\cite{xhyper_vol1}, there are no non-trivial primitive solutions to
\begin{equation*}
    x^7 + y^7 = z^p
\end{equation*}
when $p$ is sufficiently large.
\end{corollary}

\paragraph*{\bf Electronic resources} For the computations needed in this paper, we used  {\tt Magma} \cite{magma}. Our programs are posted in \cite{programs}, where there is included descriptions, output transcripts, timings, and machines used.

\paragraph*{\bf Acknowledgments} We thank Lassina Demb\'el\'e, Angelos Koutsianas and Ariel Pacetti for helpful discussions and remarks on a preliminary version. The fourth author is grateful to Max Planck Institute for Mathematics in Bonn for its hospitality and financial support.

\section{Notation and remarks}\label{S:rks}

Let~$\zeta_7$ be a primitive $7$th root of unity in~$\C$.  Recall the following standard factorization:
\begin{equation}\label{eq:basic_factorization_case_r_equal_7}
x^7+y^7=(x+y)\phi_7(x,y)=(x+y)f_1(x,y)f_2(x,y)f_3(x,y)
\end{equation}
where~$\phi_7(x,y)=x^6-x^5y+x^4y^2-x^3y^3+x^2y^4-xy^5+y^6$ and for~$i=1,2,3$,
\[
\omega_i=\zeta_7^i+\zeta_7^{-i},\quad f_i(x,y)=x^2 + \omega_i xy + y^2.
\]

The number field~$K = \Q(\zeta_7)^+ = \Q(\omega_1)$ is the cubic totally real subfield of~$\Q(\zeta_7)$. A defining polynomial for~\(K/\Q\) is given by~\(X^3 + X^2 -2X -1\) which is the minimal polynomial of~\(\omega_1\). The field~\(K\) has absolute discriminant~$49$. Let~$\calO_K$ denotes its integer ring. The unit group of~$K$ is generated by $\{-1, \epsilon_1, \epsilon_2\}$ where $\epsilon_1 = -\omega_1^2 + 1$ and $\epsilon_2 = \omega_1 + 1$.

We write $\Fq_2$, $\Fq_3$ and $\Fq_7$ for the unique prime ideals in~$K$ above~\(2\), \(3\) and~\(7\), respectively. More generaly, if~\(\ell\) is a prime, then~\(\Fq_{\ell}\) denotes a prime ideal above~\(\ell\) in~\(K\). If~\(\mathfrak{a}\) is a nonzero (integral) ideal of~\(\calO_K\), we write~\(N(\mathfrak{a})\) for its absolute norm.

If~\(\ell\) is a prime number and~\(n\) is a non-zero integer, we denote by~\(v_\ell(n)\) the \(\ell\)-adic valuation of~\(n\) at~\(\ell\); when~\(v_\ell(n) > 0\), we also write~\(\ell^{v_\ell(n)}\parallel n\).

Let~\(a,b\) be coprime integers such that~\(a + b \neq 0\). Then, \(f_1(a,b)\), \(f_2(a,b)\), and~\(f_3(a,b)\) are pairwise coprime away from~\(\Fq_7\). We have~\(\gcd(a + b, \phi_7(a,b)) = 1\) or~\(7\). Moreover, $7\mid \phi_7(a,b)$ if and only if~$7\mid a+b$, and in that case, we have $v_7(\phi_7(a,b))=1$ (see~\cite[Lemma~2.1]{DahmenSiksek}).

Let us finally make some basic remarks on equations~\eqref{E:77p} and~\eqref{main-equ}. 

The case $n = 2$ of Theorem~\ref{T:main}  follows from \cite[Theorem 1.1]{BenSki} and the case $n = 3$ from \cite[Theorem 1.5]{BenVatYaz}. The case $n = 7$ is covered by Theorem~2 in \cite[Section~4.3]{Serre87} by noticing that the proof in {\it loc.\ cit.} holds for exponents~$\geq 5$ when $L = 3$. 

The case $n = 2, 3$ of Part~\eqref{T:other_item1} in Theorem~\ref{T:other} follows from \cite{DarmonMerel}. The case~\(n = 7\) is a particular case of Fermat's last theorem whose first proof is due to Lam\'e (see~\cite{Lame}).

Therefore, in solving~\eqref{E:77p} and~\eqref{main-equ}, we can (and will) assume~$n = p$ is prime and~$p\geq 5$, $p\neq 7$.

\section{A Frey curve over $\mathbb{Q}$}
\label{S:overQ}

\subsection{Basic properties}

Let~$a$ and~$b$ be coprime integers. We consider the elliptic curve~$E_{a,b}$ defined over~$\Q$ by the equation
\begin{equation}\label{eq:EoverQ}
E_{a,b}\ : \ Y^2 = X^3 + a_2X^2 + a_4 X + a_6,
\end{equation}
where
\begin{eqnarray*}
   a_2 & = & -(a-b)^2, \\
   a_4 & = & -2a^4 + a^3 b - 5a^2 b^2 + ab^3 - 2b^4, \\
   a_6 & = & a^6 - 6a^5 b + 8a^4 b^2 - 13a^3 b^3 + 8a^2b^4 - 6ab^5 + b^6.
\end{eqnarray*}
The standard invariants attached to the model~\eqref{eq:EoverQ} are
\begin{eqnarray*}
c_4(E_{a,b}) & = & 2^4\cdot 7(a^4 - a^3b + 3a^2b^2 - ab^3 + b^4), \\
c_6(E_{a,b}) & = & 2^5\cdot 7(a^6 - 15a^5b + 15a^4b^2 - 29a^3b^3 + 15a^2b^4 - 15ab^5 + b^6), \\
\Delta(E_{a,b}) & = & 2^4\cdot7^2\phi_7(a,b)^2.
\end{eqnarray*}
This curve has been considered by Kraus in~\cite[\S4.5.1.3]{kraus5} and by Freitas (up to the change of variables given by~$X'=6^2(X-(a-b)^2/3)$ and~$Y'=6^3Y$, this is the curve denoted~$E_{(a,b)}$ in~\cite[p.~618]{F}; note the sign mistake in the definition of the coefficient~\(a_4\) in the published version of \emph{loc. cit.} though).

Let $F/\Q_\ell$ be a finite extension and $F^\unr$ its maximal unramified extension in a fixed algebraic closure of~\(\Q_\ell\). Let $E/F$ an elliptic curve with additive potential good  reduction. The extension $F^\unr(E[p])$ where $p \geq 3$ is a prime~$\neq \ell$ is independent of~$p$ and it is
the minimal extension of~$F^\unr$ where $E$ obtains good reduction. 
The degree $e = [F^\unr(E[p]) : F^\unr]$ is called the semistability defect of~$E$.    

\begin{proposition}\label{P:conductorE}
The model~\eqref{eq:EoverQ} of~$E_{a,b}$ is minimal. The conductor~$N(E_{a,b})$ of~$E_{a,b}$ is given by $N(E_{a,b})=2^{\alpha}7^2\mathrm{rad}_7(\phi_7(a,b))$, where~$\mathrm{rad}_7(\phi_7(a,b))$ is the product of all primes $\neq 7$ dividing~$\phi_7(a,b)$ and
\begin{equation}\label{eq:cond_of_E_at_2}
\alpha = \left\{ \begin{array}{ll}
2 & \quad \text{if $4 \mid ab$},\\ 
3 & \quad \text{if  $2 \parallel ab$ or $4\mid a+b$}, \\
4 & \quad \text{if $2\parallel a+b$}. \\
\end{array} \right.
\end{equation}
Moreover, we have the following properties~:
\begin{itemize}
\item the curve $E_{a,b}$ has additive potentially good reduction at~$7$
with semistability defect $e=3$ or~$6$ if~$7\mid a+b$ or~$7\nmid a+b$, respectively;
\item the curve $E_{a,b}$ has additive potentially good reduction at~$2$ with semistability 
defect $e=6$ or~$24$ if~$ab$ is even or~$ab$ is odd, respectively;
\item if~$E_{a,b}$ has bad reduction at a prime~$\ell\not=2,7$, then $\ell\equiv1\pmod{7}$ and~$v_\ell(\Delta(E_{a,b}))=v_\ell(\phi_7(a,b))=v_\ell(a^7+b^7)$.
\end{itemize}
\end{proposition}
\begin{proof}
It follows from the remarks in Section~\ref{S:rks} that the model~\eqref{eq:EoverQ} is minimal at~$7$ with
\[
\left(v_7(c_4(E_{a,b})), v_7(\Delta(E_{a,b}))\right)=(\geq 2,4)\text{ or } (\geq 1,2)
\]
according to whether~$7\mid a+b$ or~$7\nmid a+b$. Hence~$E_{a,b}$ has additive potentially good reduction at~$7$; in particular, $v_7(N(E_{a,b}))=2$. Its semistability defect~$e$ at~$7$ is equal to the denominator of~$v_7(\Delta(E_{a,b}))/12$ (cf. \cite[p.~312]{Ser72}).

We check that~$(v_2(c_4(E_{a,b})),v_2(c_6(E_{a,b})), v_2(\Delta(E_{a,b})))=(4,5,4)$. Hence the model~\eqref{eq:EoverQ} of~$E_{a,b}$ is minimal at~$2$ and~$E_{a,b}$ has additive potential good reduction at~$2$. According to~\cite[p.~358]{kraus6}, its semistability defect~$e$ at~$2$ is~$6$ or~$24$ if~$ab$ is even or~$ab$ is odd respectively. We now compute the valuation at~$2$ of~$N(E_{a,b})$. By~\cite{papado}, we are in a case~$3$, $4$ or~$5$ in Tate's algorithm and hence $v_2(N(E_{a,b}))$ is~$4$, $3$ or~$2$ respectively. If~$ab$ is even then Proposition~1 in~\emph{loc. cit.} (applied with~$(r,t)=(0,1)$) shows that we are in a case~$\ge4$. Moreover, by Proposition~2 of \emph{loc. cit.} (and in its notation), we have~$b_8\equiv -(ab)^2\pmod{8}$ and hence we are in case~$5$ if~$4\mid ab$ and in case~$4$ if~$2\parallel ab$. Similarly, if~$ab$ is odd, then Propositions~1 and~2 of~\emph{loc. cit.} (applied with~$(r,t)=(1,1)$) show that we are in case~$3$ or~$4$ if~$2\parallel a+b$ or~$4\mid a+b$ respectively.

Let~$\ell\not=2,7$ be a prime of bad reduction. Then, $\ell\mid\phi_7(a,b)$ and, since, $\ell\nmid a+b$ (recall from Section~\ref{S:rks} that~\(\phi_7(a,b)\) and~\(a + b\) are coprime away from~\(7\)), we have~$\ell\equiv 1\pmod{7}$. Moreover, we check that~$c_4(E_{a,b})=-2^4(A_1A_2+A_1A_3+A_2A_3)$ where
\[
A_1=(\omega_3-\omega_2)f_1(a,b),\quad A_2=(\omega_1-\omega_3)f_2(a,b),\quad A_3=(\omega_2-\omega_1)f_3(a,b).
\]
Since $A_1$, $A_2$ and~$A_3$ are pairwise coprime away from~$\Fq_7$ and~$\ell$ divides~$A_1A_2A_3=7\phi_7(a,b)$, we see that~$E_{a,b}$ has bad multiplicative reduction at~$\ell$.
\end{proof}

Let~$p$ be a prime. Write~\(\rhobar_{E_{a,b},p}\) for the mod~\(p\) representation attached to~\(E_{a,b}/\Q\).

\begin{lemma}\label{lem:irredE}
Assume~$p\ge5$. Then the representation $\rhobar_{E_{a,b},p}$ is irreducible.
\end{lemma}
\begin{proof} According to Proposition~\ref{P:conductorE}, the curve~$E_{a,b}$ has  additive potential good reduction at~$2$ with semistability defect $e=6$ or~$24$. It then follows from~\cite[Corollaire~3.4]{billereyIrred} that the representation~$\rhobar_{E_{a,b},p}$ is irreducible when $p \ge 5$.
\end{proof}

\subsection{Consequences on the solutions}

Let~\((a,b,c) \in \Z^3\) be a non-trivial primitive solution to~\eqref{E:77p} with exponent~\(n = p\) prime, \(p\geq 5\), \(p\neq7\) (see Section~\ref{S:rks}). Then, \(a\) and~\(b\) are coprime. Write~$E=E_{a,b}$ for simplicity and denote by~$N(\bar{\rho }_{E,p})$ the Serre level (i.e. the Artin conductor away from~$p$) of~$\bar{\rho}_{E,p}$. Since~$p \neq 7$, by Proposition~\ref{P:conductorE} and~\cite{kraus7}, then~$N(\bar{\rho}_{E,p})=2^\alpha\cdot7^2$ where~$\alpha$ is defined in~\eqref{eq:cond_of_E_at_2}.

For an integer $M > 0$, let $S_2(M)$ denote the set of cuspforms of weight~\(2\), trivial character and level~$M$. By the classical modularity and level-lowering theorems, there exist a newform $f \in S_2\left(2^\alpha\cdot7^2\right)$ and a prime~$\Fp \mid p$ in~$\Qbar$ such that
\begin{equation}
\bar{\rho}_{E,p} \simeq \bar{\rho}_{f,\Fp},
\label{E:iso}
\end{equation}
where~\(\bar{\rho}_{f,\Fp}\) denotes the reduction of the \(\Fp\)-adic representation attached to~\(f\).

The following theorem shows that this isomorphism cannot hold in many cases.

\begin{theorem}\label{T:overQ}
We are in one of the following situations:
\begin{enumerate}
 \item $4\mid ab$, \(7 \nmid a + b\) and~$\bar{\rho}_{E,p}\simeq\bar{\rho}_{E_{1,0},p}$~;
 \item \(4 \nmid ab\), \(7 \mid a + b\) and~$\bar{\rho}_{E,p}\simeq\bar{\rho}_{E_{1,-1},p}$.
\end{enumerate}

In particular, equation~\eqref{E:77p} does not have any non-trivial primitive solution~$(a,b,c)$ with~$ab$ odd unless~$7\mid a + b$. 
\end{theorem}
\begin{proof} Note that when $2 \mid a+b$ then $4 \mid a+b$ due to the shape of \eqref{E:77p}. In particular, from Proposition~\ref{P:conductorE} it follows that the conductor of $E$ at~$2$ is either $2^2$ or $2^3$.

From Proposition~\ref{P:conductorE}, for every prime~$q\not=2,7,p$ satisfying~$q\not\equiv 1\pmod{7}$ the curve ~$E$ has good reduction at~\(q\). Then, from 
the isomorphism~\eqref{E:iso}, there exists~$(x,y)\in\{0,\dots,q-1\}^2\backslash\{(0,0)\}$ such that~\(a_q(f)\equiv a_q(E_{x,y})\pmod{\Fp}\). 

Using these congruences for all such primes~$q\leq40$, we are able to contradict the isomorphism~\eqref{E:iso} for all newforms~$f$, except those corresponding to the elliptic curves~$E_{1,0}$ and~$E_{1,-1}$ in levels~$2^2\cdot7^2$ and~$2^3\cdot7^2$, respectively. These computations only take few seconds (see~\cite{programs} for details).

Now the statement about $7$-adic condition follows directly from the computation of the semistability defect at~$7$ in Proposition~\ref{P:conductorE}.
\end{proof}

\section{A Frey curve over a totally real cubic field}
\label{S:overCubic}

In this section, we extend the methods of~\cite{F}. The information provided here is used in different ways for our various proofs of Theorem~\ref{T:main} in the following sections.

Let~$a$ and~$b$ be coprime integers with~$a+b\not=0$. We use the notation of Section~\ref{S:rks} and consider the Frey curve $F_{a,b} \coloneq E_{(a,b)}^{(1,2)}$ as defined in \cite[p.~619]{F} given by the model
\begin{equation}\label{eq:FoverK7}
F_{a,b}\ : \ Y^2 = X(X-A_{a,b})(X+B_{a,b}),
\end{equation}
where
\begin{eqnarray*}
   A_{a,b} & = & (\omega_2 - \omega_1)(a + b)^2 \\
   B_{a,b} & = & (2 - \omega_2)(a^2 + \omega_1ab + b^2).
\end{eqnarray*}
The standard invariants of model~\eqref{eq:FoverK7} are
\begin{eqnarray*}
c_4(F_{a,b}) & = & 2^4(A_{a,b}^2 + A_{a,b}B_{a,b} + B_{a,b}^2), \\
c_6(F_{a,b}) & = & 2^5(2A_{a,b}^3 + 3A_{a,b}^2B_{a,b} - 3A_{a,b}B_{a,b}^2 - 2B_{a,b}^3), \\
\Delta(F_{a,b}) & = & 2^4\left(A_{a,b}B_{a,b}C_{a,b}\right)^2,
\end{eqnarray*}
where
\[
C_{a,b}\coloneq -(A_{a,b} + B_{a,b})=(\omega_1 -2)(a^2 + \omega_2ab + b^2).
\]
Note that $F_{a,b}$ has full $2$-torsion defined over $K$. 

Write $F = F_{a,b}$. For~$\delta\in \calO_K\backslash\{0\}$, a Weierstrass model for the quadratic twist~$F^{(\delta)}$ of~$F$ by~$\delta$ is given by
\begin{equation}\label{eq:FtwistedoverK7}
y^2 = x(x - \delta A_{a,b})(x + \delta B_{a,b}).
\end{equation}
Its standard invariants are given by
\[
c_4\big(F^{(\delta)}\big) = \delta^2 c_4(F),\qquad c_6\big(F^{(\delta)}\big) = \delta^3 c_6(F),\qquad \Delta\big(F^{(\delta)}\big) = \delta^6 \Delta(F).
\]
For a prime ideal~$\Fq$ in~$K$ and an ideal~$\mathfrak{a}$, we denote by~$v_\Fq(\mathfrak{a})$ the valuation at~$\Fq$ of~$\mathfrak{a}$. For an element~$\alpha\in\calO_K\backslash\{0\}$, we write~$v_\Fq(\alpha)$ for~$v_\Fq(\alpha\calO_K)$.  
\begin{proposition} 
Let~$\delta\in\calO_K$. Denote by~$N(F^{(\delta)})$ the conductor of~$F^{(\delta)}$. 
\begin{itemize}
\item If~$\delta$ is coprime to~$\Fq_7$, then we have
\begin{equation*}
v_{\Fq_7}\left(N(F^{(\delta)})\right) = \left\{ \begin{array}{ll}
1 & \quad \text{if $7 \mid a + b$},\\ 
2 & \quad \text{if $7 \nmid a + b$}.\\ 
\end{array} \right.
\end{equation*}

\item If~$\delta\equiv 1\pmod{\Fq_2^2}$, then we have
\begin{equation*}
v_{\Fq_2}\left(N(F^{(\delta)})\right) = \left\{ \begin{array}{ll}
0 & \quad \text{if $4 \parallel a + b$},\\ 
1 & \quad \text{if $8 \mid a + b$},\\ 
3 & \quad \text{if  $4 \mid ab$}, \\
4 & \quad \text{if $2\parallel a+b$ or~$2\parallel ab$}. \\
\end{array} \right.
\end{equation*}

\item If~$2\parallel ab$ and~$\delta\equiv \omega_2 \pmod{\Fq_2^2}$, then we have~$v_{\Fq_2}\left(N(F^{(\delta)})\right) = 3$.

\item If~$7 \nmid a + b$ and~$v_{\Fq_7}(\delta)$ is odd, then we have~$v_{\Fq_7}\left(N(F^{(\delta)})\right) = 0$.
\end{itemize}
Moreover, a prime ideal~$\Fq\nmid \Fq_2\cdot\Fq_7\cdot\delta$ is a prime of bad reduction for~$F^{(\delta)}$ if and only if it divides $(a + b)(a^2 + \omega_1ab + b^2)(a^2 + \omega_2ab + b^2)$. In that case, $F^{(\delta)}$ has bad multiplicative reduction at~$\Fq$ and we have
\[
v_\Fq\big(\Delta(F^{(\delta)})\big) = 
\left\{
\begin{array}{ll}
4v_\Fq(a^7 + b^7) & \text{if $\Fq\mid a + b$}, \\
2v_\Fq(a^7 + b^7) & \text{otherwise}. \\
\end{array}
\right.
\]
\label{P:conductorF}
\end{proposition}
\begin{proof} For simplicity, write~$A = A_{a,b}$, $B = B_{a,b}$, $C = C_{a,b}$ and denote by~$(c_4,c_6,\Delta)$ and $\big(c_4^{(\delta)},c_6^{(\delta)},\Delta^{(\delta)}\big)$ the standard invariants of~$F$ and~$F^{(\delta)}$ respectively.
We have that~$(a + b)^2$, $f_1(a,b)$, $f_2(a,b)$ and~$f_3(a,b)$ are pairwise coprime away from~$\Fq_7$ and so are~$A$, $B$ and~$C$ since~$\omega_2 - \omega_1$, $2 - \omega_2$ and~$\omega_1 -2$ all generate the prime ideal~$\Fq_7$. 

Let~$\Fq$ be a prime ideal such that $\Fq\nmid \Fq_2\cdot\Fq_7\cdot\delta$. Then~$\Fq$ divides~$\Delta^{(\delta)}$ if and only if it divides~$(a + b)(a^2 + \omega_1ab + b^2)(a^2 + \omega_2ab + b^2)$. In that case, $\Fq$ does not divide~$c_4^{(\delta)}$. Hence, $F^{(\delta)}$ has bad multiplicative reduction at~$\Fq$ and
\begin{align*}
v_\Fq\big(\Delta^{(\delta)}\big) = v_\Fq(\Delta)  & = 2v_{\Fq}\left((a+b)^2(a^2 + \omega_1ab + b^2)(a^2 + \omega_2ab + b^2)\right) \\
									 & =\left\{
									 \begin{array}{ll}
									4v_\Fq(a^7 + b^7) & \text{if $\Fq\mid a + b$} \\
									2v_\Fq(a^7 + b^7) & \text{otherwise} \\
									\end{array}
									 \right.
\end{align*}
according to the factorization~\eqref{eq:basic_factorization_case_r_equal_7}.

Let us compute the valuation of the conductor at~$\Fq_7$ under the assumptions of the statement.

Assume~$7\mid a+b$ and $\delta$ is coprime to~$\Fq_7$. Then, we have~$v_{\Fq_7}(A) = 1 + 6v_7(a + b)$ and~$v_{\Fq_7}(B) = 1$. Therefore, we have 
\[
\big(v_{\Fq_7}\big(c_4^{(\delta)}\big),v_{\Fq_7}\big(\Delta^{(\delta)}\big)\big) = \big(v_{\Fq_7}(c_4),v_{\Fq_7}(\Delta)\big) = \big(4,10+12v_7(a+b)\big).
\] 
In particular, the model~\eqref{eq:FtwistedoverK7} is not minimal at~$\Fq_7$, the curve~$F^{(\delta)}$ has bad multiplicative reduction at~$\Fq_7$ and  we have~$v_{\Fq_7}\left(N(F^{(\delta)})\right)=1$. 

Assume~$7\nmid a+b$. Then we have~$v_{\Fq_7}(A) = v_{\Fq_7}(B) = v_{\Fq_7}(C) = 1$. In particular, we have~\(v(c_4)\ge2\). If moreover~$\delta$ is coprime to~$\Fq_7$, then we have
\[
\big(v_{\Fq_7}\big(c_4^{(\delta)}\big), v_{\Fq_7}\big(\Delta^{(\delta)}\big)\big) = \big(v_{\Fq_7}(c_4), v_{\Fq_7}(\Delta)\big) = (\ge 2, 6).
\]
In particular, the model~\eqref{eq:FtwistedoverK7} is minimal at~$\Fq_7$, the curve~$F^{(\delta)}$ has bad additive reduction at~$\Fq_7$ and we have~$v_{\Fq_7}\left(N(F^{(\delta)})\right) = 2$.  

If instead we have~$v_{\Fq_7}(\delta)$ odd, then we have
\begin{align*}
\big(v_{\Fq_7}\big(c_4^{(\delta)}\big),v_{\Fq_7}\big(\Delta^{(\delta)}\big)\big) 
& = \big(2v_{\Fq_7}(\delta) + v_{\Fq_7}(c_4), 6v_{\Fq_7}(\delta) + v_{\Fq_7}(\Delta)\big) \\
& = \big(2v_{\Fq_7}(\delta) + v_{\Fq_7}(c_4), 6(v_{\Fq_7}(\delta) + 1)\big).
\end{align*}
Write~\(v_{\Fq_7}(\delta) = 2k + 1\). We have
\[
\big(v_{\Fq_7}\big(c_4^{(\delta)}\big),v_{\Fq_7}\big(\Delta^{(\delta)}\big)\big) = \big(v_{\Fq_7}(c_4) + 4k + 2, 12(k + 1)\big).
\]

Therefore, the model~\eqref{eq:FtwistedoverK7} is not minimal at~$\Fq_7$ (see~\cite[Tableau~I]{papado}). Applying the substitution $x \rightarrow x/u^{2(k + 1)}$, $y\rightarrow y/u^{3(k + 1)}$ with~\(u = 2 + \omega_1 - \omega_2\) (a generator of~\(\Fq_7\)) shows that the curve~$F^{(\delta)}$ has good reduction at~$\Fq_7$. Hence we have~$v_{\Fq_7}\left(N(F^{(\delta)})\right) = 0$.  

Let us now compute the valuation of the conductor at~$\Fq_2$ under the assumptions of the statement.

Assume first that~$2\nmid ab$ and~$\delta\equiv 1\pmod{\Fq_2^2}$. Then, we have $2\mid a + b$, $v_{\Fq_2}(A) = 2v_2(a + b)$ and~$v_{\Fq_2}(B) = v_{\Fq_2}(C) = 0$. Moreover~$\delta$ is coprime to~$\Fq_2$ and hence we have 
\begin{align*}
\big(v_{\Fq_2}\big(c_4^{(\delta)}\big),v_{\Fq_2}\big(c_6^{(\delta)}\big),v_{\Fq_2}\big(\Delta^{(\delta)}\big)\big) & = \big(v_{\Fq_2}(c_4),\, v_{\Fq_2}(c_6),\, v_{\Fq_2}(\Delta)\big) \\
& = \big(4,\, 6,\, 4+4v_2(a+b)\big). 
\end{align*}

If~$v_2(a+b)=1$, then the model~\eqref{eq:FtwistedoverK7} corresponds to a case $6$, $7$ or $8$ in Tate's classification with corresponding valuation~$v_{\Fq_2}\left(N(F^{(\delta)})\right) = 4$, $3$ or~$2$ (see~\cite[Tableau~IV]{papado}). Applying~\cite[Proposition~3]{papado} with, in its notation, $r = -2\omega_1$ and $t = -2\omega_1(1 - \omega_1)$ we get that we are in Case~$6$ in Tate's classification. Hence we have~$v_{\Fq_2}\left(N(F^{(\delta)})\right) = 4$ in this case. 

Similarly, if~$v_2(a+b)\ge2$, then the model~\eqref{eq:FtwistedoverK7} corresponds to either a case $7$ in Tate's classification (with Kodaira type $\mathrm{I}_\nu^*$ where $\nu = 4(v_2(a + b) - 1)\geq4$) or a non-minimal case (see~\cite[Tableau~IV]{papado}). We apply~\cite[Proposition~4]{papado} with $r = 0$ and~$s = 1 - \omega_1$ to conclude that we are in a non-minimal case. Therefore we have good reduction if~$v_2(a+b)=2$ and bad multiplicative reduction otherwise.

Assume now~$2\mid ab$ and~$\delta$ coprime to~$\Fq_2$. Then, $A$, $B$ and $C$ reduce to~$\omega_1 + \omega_2$, $\omega_2$ and~$\omega_1$ modulo~$\Fq_2$ respectively and we have
\[
\frac{c_4}{2^4}\equiv A^2 + AB + B^2\not\equiv0\pmod{\Fq_2}\quad\text{and}\quad\frac{c_6}{2^5}\equiv ABC\not\equiv0\pmod{\Fq_2}.
\]
Hence, we have 
$$
\big(v_{\Fq_2}\big(c_4^{(\delta)}\big),v_{\Fq_2}\big(c_6^{(\delta)}\big),v_{\Fq_2}\big(\Delta^{(\delta)}\big)\big) = \big(v_{\Fq_2}\big(c_4\big),v_{\Fq_2}\big(c_6\big),v_{\Fq_2}\big(\Delta\big)\big) = \big(4,\, 5,\, 4\big).
$$
Therefore, we are in Case~$3$, $4$ or $5$ of Tate's classification (see~\cite[Tableau~IV]{papado}) with corres\-ponding valuation~$v_{\Fq_2}\left(N(F^{(\delta)})\right) = 4$, $3$ or~$2$.

If~$\delta\equiv 1\pmod{\Fq_2^2}$, then applying~\cite[Proposition~1 and~2]{papado} with~$(r,t) = (1 - \omega_1 + \omega_1^2,-\omega_1)$, we get that we are in Case~$3$ or in Case~$4$ of Tate's classification if~$2\parallel ab$ or $4\mid ab$ respectively.

If~$2\parallel ab$ and~$\delta \equiv z^2 - 2\pmod{\Fq_2^2}$, applying~\cite[Proposition~1 and~2]{papado} with~$(r,t) = (-\omega_1 + \omega_1^2,1 - \omega_1 + \omega_1^2)$, we get that we are in Case~$4$ in Tate's classification. Hence we have~$v_{\Fq_2}\left(N(F^{(\delta)})\right) = 3$.
\end{proof}

\begin{remark}\label{rem:twist}
From Proposition~\ref{P:conductorF} we see that under some $2$-adic or $7$-adic conditions taking a quadratic twist of~$F$ by an appropriate~$\delta$ yields $F_{a,b}^{(\delta)}$ with smaller conductor. This conductor reduction  is essential to make the modular method using~$F_{a,b}$ computationally feasible
in Section~\ref{S:Frey7first}. 
Since this observation is likely to be useful in future applications of the modular method let
us briefly describe how we found 
the twist~$\delta = \omega_2$ in Table~\ref{table:delta} which according to Proposition~\ref{P:conductorF} reduces the conductor exponent at~$\Fq_2$ from 4 to~$3$ when~$2 \Vert ab$. Indeed, a standard calculation shows the semistability defect of $F_{1,1}$ at~$\Fq_2$ is $e=24$,
and so $F_{1,1}$  obtains good reduction 
over an extension $L/K_{\Fq_2}^{\unr}$ such that $\Gal(L/K_{\Fq_2}^{\unr}) \simeq \SL_2(\F_3)$,  and $F_{1,1}$ has an exceptional supercuspidal inertial type at~$\Fq_2$. In particular, no quadratic twist will reduce the size of inertia, however these exceptional types tend to appear (up to twist) at odd conductor exponents or large even conductor 
exponents (see \cite[Appendix]{dia95} for the case over~$\Q_2$) and since $v_{\Fq_2}\left(N(F_{1,1})\right) = 4$ we might still find a twist having lower odd conductor exponent. 
With the help of {\tt Magma}~\cite{magma} we can compute the $2$-Selmer group of~$K$ unramified outside~$2$ which allows to recover all quadratic extensions of $K$ ramified only at~$\Fq_2$. Each such extension corresponds to a possible~$\delta$ and we check the conductor of~$F_{1,1}^{(\delta)}$ until it lowers.
\end{remark}

Let~$p$ be a prime. Recall that~\(F = F_{a,b}\) and write~\(\rhobar_{F,p} : G_K\rightarrow\GL_2(\F_p)\) for the mod~\(p\) representation attached to~\(F/K\), where~\(G_K\) denotes the absolute Galois group of~\(K\).

The following result generalizes~\cite[Theorem~11.2]{F}.
\begin{proposition} \label{P:irredOverK}
Assume that~$p = 5$ or $p \geq 11$. Then, the representation $\rhobar_{F,p}$ is irreducible.
\end{proposition}
\begin{proof}  Suppose $\rhobar_{F,p}$ is reducible, that is,
\[ \rhobar_{F,p} \sim \begin{pmatrix} \theta & \star\\ 0 & \theta' \end{pmatrix} 
\quad \text{with} \quad \theta, \theta' : G_K \rightarrow \F_p^* 
\quad \text{satisfying} \quad \theta \theta' = \chi_p.\]

The characters $\theta$ and $\theta'$ ramify only at $p$ and additive primes of $F$; the latter can be
$\Fq_{7}$ and $\Fq_2$ according to Proposition~\ref{P:conductorF}. When these primes are 
of potentially good reduction, \cite[Theorem~5.1]{Jarv2} implies there is no conductor degeneration mod~$p$. Moreover, in case~$\Fq_2$ or~$\Fq_7$ is of additive potentially multiplicative reduction, it follows from the theory of the Tate curve that 
the conductor of $\rhobar_{F,p}$ at such prime is the same as that of~$F$ (even when $\star = 0$ on inertia).
Therefore, at an additive prime $\Fq$ 
both $\theta$, $\theta'$ have conductor 
exponent equal to $\vv_\Fq(N_F)/2$.
We conclude that the possible conductors 
away from~$p$ for 
$\theta$ and $\theta'$ are divisors of $\Fq_2^2 \Fq_{7}$.

Recall from Section~\ref{S:rks} that the unit group of~$K$ is generated by $\{-1, \epsilon_1, \epsilon_2\}$ where $\epsilon_1 = -\omega_1^2 + 1$ and $\epsilon_2 = \omega_1 + 1$. In the notation of \cite[Theorem~1]{FS}, we compute $B = 7 \cdot 13^3$. Thus from the first paragraph of the proof of \cite[Theorem~1]{FS} we conclude that for $p=11$ and $p \geq 17$ exactly one of $\theta$, $\theta'$ ramifies at~$p$. So, replacing $F$ by a $p$-isogenous curve if necessary, we can assume $\theta$ is unramified at $p$. We conclude that the conductor of $\theta$ divides $\Fq_2^2 \Fq_{7}$.

Let $\infty_1$, $\infty_2$ and $\infty_3$ be the real places of $K$.
The ray class group for modulus $\Fq_2^ 2 \Fq_{7} \infty_1 \infty_2 \infty_3$
is isomorphic to
\[
\Z/2\Z \oplus  \Z/2\Z \oplus \Z/2\Z \oplus \Z/2\Z.
\]
Thus $F$ has a quadratic twist with a $p$-torsion point defined over $K$.  From the proof of \cite[Theorem 4]{BruNaj}, we see that this is impossible for $p > 13$, concluding the proof in this case.

It remains to deal with~\(p = 5\), \(11\) and~\(13\).

For $p = 5$, the elliptic curve~\(F\) gives rise to a $K$-rational point on~$X_0(20)$ since $F$ has full $2$-torsion over~$K$. By the proof of \cite[Lemma 10.4]{AnniSiksek}, the $K$-rational points on $X_0(20)$ are cuspidal. Hence the result for~\(p = 5\).

For $p = 11$, the elliptic curve~$F$ gives rise to a $K$-rational point on $X_0(11)$ and by \emph{loc. cit.} it corresponds to an elliptic curve with integral $j$-invariant. According to the last statement of Proposition~\ref{P:conductorF}, it then follows that~\(a^7 + b^7 = \pm 2^u\cdot7^v\) for some non-negative integers~\(u,v\). Since~\(\phi_7(a,b)\) is odd and satisfies~\(v_7(\phi_7(a,b))\leq 1\) (see Section~\ref{S:rks}), we get that
\[
\phi_7(a,b) = \pm1\quad\text{or}\quad \phi_7(a,b) = \pm7.
\]
As~\(a + b \neq 0\) and \(\phi_7(a,b) = \frac{a^7 + b^7}{a + b}\) only takes positive values, it follows from~\cite[Proposition~2.5]{F} that
\[
(a,b)\in\{\pm(1,1), (\pm1,0), (0,\pm1)\}.
\]
Let~\((a,b)\) be one of these six pairs. We check that the elliptic curve~\(F_{a,b}\) has good reduction at~\(\Fq_3\). Moreover, we have~\(a_{\Fq_3}(F_{a,b})  = \pm4\) and hence the characteristic polynomial~\(X^2 \pm4X + 27\) of~\(\overline{\rho}_{F_{a,b},11}(\Frob_{\Fq_3})\) in~\(\F_{11}[X]\) is irreducible. Therefore, none of the curves~\(F_{a,b}\) for~\((a,b)\) as above, has a \(K\)-rational isogeny of degree~\(11\). This rules out the case~\(p = 11\).

Finally, suppose $p = 13$. If one of $\theta$ or $\theta'$ is unramified at all the primes in $K$ above $13$, 
replacing $F$ by a $13$-isogenous curve if necessary, we can assume 
again the conductor of $\theta$ divides $\Fq_2^2 \Fq_{7}$. Arguing with ray class groups as above, we conclude that, up to quadratic twist, $F$ has a $13$-torsion point over~$K$; since it also 
has full $2$-torsion this gives rise to a $K$-rational point on $X_1(26)$.
By the proof of \cite[Lemma 10.4]{AnniSiksek}, this is not possible.
Then, both $\theta$ and $\theta'$ ramify at some prime above $13$. 

Let $\Fp_1$, $\Fp_2$, $\Fp_3$ denote the primes of $K$ dividing~$13$.
Since $13$ is unramified in $K$ and $F$ is semistable at the $\Fp_i$ 
it follows that one of  $\theta$, $\theta'$ ramifies at one $\Fp_i$ and the other 
at the other two primes (see \cite[Lemme~1]{Kraus3} and its proof). 
So, replacing $F$ by a $13$-isogenous curve if necessary, we may assume $\theta$ 
is ramified at exactly one prime above 13, say $\Fp_j$. Furthermore, since $F$ has either 
multiplicative or good ordinary reduction at $\Fp_j$, we also know that 
$\theta|_{I_{\Fp_j}} = \chi_{13}|_{I_{\Fp_j}}$. Thus the field cut out by $\theta$ has 
modulus dividing $\Fq_2^2 \Fq_7 \Fp_{j}$.

The ray class groups of modulus $\Fq_2^2 \Fq_7 \Fp_{1}^{e_1} \Fp_{2}^{e_2} \Fp_{3}^{e_3} \infty_1 \infty_2 \infty_3$, where $e_1 + e_2 + e_3 = 1$ are all isomorphic to $(\Z/2\Z)^5$. Therefore, up to a quadratic twist, 
$F$ gives again a $K$-rational point on $X_1(26)$, which is ruled out as above.

For details on the computations used in this proof, see~\cite{programs}.
\end{proof}

\section{A first proof of Theorem~\ref{T:main} using twists of Frey curves} 
\label{S:Frey7first}

In this section, we prove Theorem~\ref{T:main} using solely the Frey curve~\(F\) defined in the previous section and its twists. In Remark~\ref{rk:EandF}, we also discuss how the information from the Frey curve~$E$ obtained in Theorem~\ref{T:overQ} can be used to give an alternative proof of about the same total running time but with some memory-saving (see Remark~\ref{rk:EandF}).

Let~\((a,b,c)\) be a non-trivial primitive solution to equation~\eqref{E:77p} for exponent~\(n = p\) prime, \(p\ge 5\), \(p\neq 7\) (see Section~\ref{S:rks}). Then, \(a\) and~\(b\) are non-zero coprime integers such that~\(a + b\neq 0\).

Keeping the notation of Sections~\ref{S:rks} and~\ref{S:overCubic}, write again~$F=F_{a,b}$ for simplicity and define~$\delta\in\calO_K\backslash\{0\}$ as in Table~\ref{table:delta}. Note that~$\delta$ is either a unit  or $7$ times a unit in~$\calO_K$.

\begin{table}[h]
\begin{tabular}{|c|c|c|}
\hline
 & $7\nmid a + b$  & $7\mid a + b$ \\
\hline
$2\nmid ab$ & $-7$ & $1$ \\
\hline
$2\parallel ab$ & $-7\omega_2$ & $\omega_2$ \\
\hline
$4\mid ab$ & $-7$ & $1$ \\
\hline
\end{tabular}
\caption{Values of~$\delta$}\label{table:delta}
\end{table}
Let~$N(\bar{\rho}_{F^{(\delta)},p})$ be the Serre level of the mod~$p$ representation~$\bar{\rho}_{F^{(\delta)},p}$ associated with~$F^{(\delta)}$ (i.e., the prime-to-$p$ part of the Artin conductor). Note that when $2 \mid a+b$ then $8 \mid a+b$ due to the shape of~\eqref{E:77p}.  According to Proposition~\ref{P:conductorF} (recall that $3 \mid a+b$ and $p \neq 7$), we have~$N(\bar{\rho}_{F^{(\delta)},p}) = \Fq_2^s \Fq_3 \Fq_7^t$ with 
\begin{equation}\label{eq:cond_of_F_delta_at_2_and_7}
s = \left\{ \begin{array}{ll}
1 & \quad \text{if $2 \nmid ab$},\\ 
3 & \quad \text{if  $2 \mid ab$}, \\
\end{array} \right.
\quad\text{and}\quad
t = \left\{ \begin{array}{ll}
0 & \text{if $7 \nmid a + b$}, \\
1& \text{if $7 \mid a + b$}. \\
\end{array} \right.
\end{equation}

Let $S_2(N(\bar{\rho}_{F^{(\delta)},p}))$ denote the space of Hilbert cuspforms of level~$N(\bar{\rho}_{F^{(\delta)},p})$, parallel weight~2 and trivial character. The curve $F$ is modular by~\cite[Corollary 6.4]{F} and hence so is~$F^{(\delta)}$. (Note that we now know that all elliptic curves over totally real cubic fields are modular by the work of Derickx, Najman, and Siksek \cite{DNS20}.) Moreover,  by Proposition~\ref{P:irredOverK}, the representation~$\rhobar_{F,p}$ is irreducible and hence so is~$\rhobar_{F^{(\delta)},p}$. An application of level lowering theorems for Hilbert modular forms (see~\cite{Fuj,Jarv,Raj}), implies that there exists a Hilbert newform $g \in S_2(N(\bar{\rho}_{F^{(\delta)},p}))$ 
such that for a prime~$\Fp \mid p$ in~$\Qbar$ 
we have
\begin{equation}
\rhobar_{F^{(\delta)},p} \simeq \rhobar_{g,\Fp},
\label{E:iso2}
\end{equation}
where~\(\bar{\rho}_{g,\Fp}\) denotes the reduction of the \(\Fp\)-adic representation attached to~\(g\).

In the sequel, we show that this isomorphism is impossible, hence proving Theorem~\ref{T:main}.

Let $q\not=2,3,7,p$ be a rational prime.

From Proposition~\ref{P:conductorF} and~\eqref{E:iso2}, there exists~$(x,y)\in\{0,\dots,q-1\}^2\backslash\{(0,0)\}$ such that 
$(a,b) \equiv (x,y) \pmod{q}$
and, for all prime ideals~$\Fq$ above~$q$ in~$K$, we have
\begin{equation}\label{eq:congruences}
a_\Fq(g)\equiv\left\{
\begin{array}{lll}
a_\Fq\big(F_{x,y}^{(\delta)}\big) & \pmod{\Fp} & \text{if $F_{x,y}^{(\delta)}$ has good reduction at~$\Fq$}, \\
\pm \left(N(\Fq)+1\right) & \pmod{\Fp} & \text{otherwise}.
\end{array}
\right.
\end{equation}
If we denote by~$\mathcal{O}$ the integer ring of the coefficient field of~$g$ and by~$b_\Fq^{(\delta)}(x,y)$ the right hand side of this congruence, then in particular
\begin{equation}\label{eq:divisibility}
p\mid \prod_{\substack{0\leq x,y\leq q-1 \\ (x,y)\not=(0,0)}}\gcd\left(\Norm_{K_g/\Q}\left((a_\Fq(g)-b_\Fq^{(\delta)}(x,y))\mathcal{O}~;\ \Fq\mid q\right)\right).
\end{equation}
We split the proof of Theorem~\ref{T:main} into two parts according to whether $7$ divides~$a + b$ or not. Part~\ref{item:partA} and item~(\ref{item:partBi}) require less than a minute of computations in total. For items~(\ref{item:partBii}) and~(\ref{item:partBiii}) the relevant space of newforms is the same. Its initialization in~\texttt{magma} requires about~\(35\) minutes, whereas the elimination procedure  takes approximatively~\(1.3\) minute in each case (see~\cite{programs} for the code and output files).

\begin{enumerate}[A.]
\item\label{item:partA} Assume~$7\nmid a + b$. Then, $t = 0$ and hence we have~$N(\bar{\rho}_{F^{(\delta)},p}) = \Fq_2\Fq_3$ or $\Fq_2^3\Fq_3$ if $2\nmid ab$ or~$2\mid ab$ respectively. Notice that when~$2\mid ab$, the values of~$\delta$ depends on the valuation at~$2$ of~$ab$ though.
\begin{enumerate}[(i)]
\item\label{item:partAi} When~$2\nmid ab$, we have~$\delta = -7$ and the form~$g$ has level~$\Fq_2 \Fq_3$. There are two newforms in~$S_2(\Fq_2 \Fq_3)$. Using the divisibility relation~\eqref{eq:divisibility} with prime ideals above $q = 5$, $13$, and $29$ we eliminate both newforms.

\item When~$2\parallel ab$, we have~$\delta = -7\omega_2$ and the form~$g$ has level~$\Fq_2^3 \Fq_3$. There are $47$ newforms in~$S_2(\Fq_2^3 \Fq_3)$. Using the divisibility relation~\eqref{eq:divisibility} with prime ideals above $q = 13$, $29$, and~$41$ we eliminate all these newforms, except the exponent $p = 5$ which survives for one form.  
We discard it using instead the congruences~\eqref{eq:congruences} themselves (a method we refer to as `refined elimination' in~\cite{BCDF2}) with auxiliary prime $q = 13$.

\item\label{item:partAiii} When~$4\mid ab$, we have~$\delta = -7$ and the form~$g$ has again level~$\Fq_2^3 \Fq_3$. Using the divisibility relation~\eqref{eq:divisibility} with prime ideals above $q = 5$, $13$, $29$, and~$41$ we eliminate all the $47$ newforms in~$S_2(\Fq_2^3 \Fq_3)$ for $p \ge 5$, except the exponent $p = 5$ which survives for one form. 
 We discard it using again refined elimination with auxiliary primes $q = 13$.
\end{enumerate}
Therefore we have proved the result when~$7$ does not divide~$a + b$. 

\item\label{item:partB} Let us now assume~$7\mid a + b$. 
\begin{enumerate}[(i)]
\item\label{item:partBi} When~$2\nmid ab$, we have~$\delta = 1$ and the form~$g$ has level~$\Fq_2 \Fq_3\Fq_7$. There are five newforms in $S_2(\Fq_2 \Fq_3 \Fq_7)$. Using the divisibility relation~\eqref{eq:divisibility} with prime ideals above $q = 5$, $13$, $29$, and~$41$ we 
eliminate all the five newforms for $p \geq 7$ and $p \neq 13$. The exponents $p=5,13$ survive simultaneously for one form $\Ff$
with field of coefficients $\Q_{\Ff}$ equal to the real cubic subfield of $\Q(\zeta_{13})$. 
Using refined elimination with auxiliary prime $q = 29$ we eliminate~$\Ff$ for $p=5$. An application of \cite[Theorem~2.1]{Martin} shows there is a 
newform $f \in S_2(\Fq_2\Fq_3\Fq_7)$ and a prime $\Fp' \mid 13$ in its coefficient
field such that $a_\Fq(f) \equiv \Norm(\Fq) + 1 \pmod{\Fp'}$ for all $\Fq \nmid 2\cdot 3 \cdot 7 \cdot 13$.
By comparing the traces for any prime ideal~$\Fq$ above~$29$, we conclude that $f=\Ff$ and $\Fp' = \Fp$ is the unique ideal in $\Q_\Ff$ above~13.
Therefore, the representation $\rhobar_{\Ff,\Fp}$ is reducible and, 
since there is no other ideal above~13 in~$\Q_\Ff$, we cannot have \eqref{E:iso2}
with $g = \Ff$ when $p=13$. (For an argument of this type in a more classical setting, see~\cite[Ex.~2.9]{HalberstadtKraus}.) This completes the proof when~$2\nmid ab$. 

\item\label{item:partBii} When~$2\parallel ab$, we have~$\delta = \omega_2$ and the form~$g$ is one the $121$ newforms in the \(818\)-dimensional new subspace of~$S_2(\Fq_2^3 \Fq_3 \Fq_7)$. We perform standard elimination for all the possible forms using the primes above $q = 5$, $13$, \(29\), and $41$. We discard this way all the newforms for $p \geq 7$. The exponent $p = 5$ survives for two forms 
and refined elimination with $q = 83$ (for one of them) 
or~$q = 29$ (for the other one) 
deals with it. 

\item\label{item:partBiii} When~$4\mid ab$, we have~$\delta = 1$ and the form~$g$ has again level~$\Fq_2^3 \Fq_3\Fq_7$. Standard elimination using the primes above $q = 5$, $13$, \(29\), and $41$ eliminates all the possible forms except three forms for the exponent~$p= 5$  
which are dealt with using refined elimination with~$q = 83$, $41$, and~$29$ respectively.
\end{enumerate}
Therefore we have proved the result for~$7\mid a + b$ as well. 
\end{enumerate} 
This finishes the proof of Theorem~\ref{T:main} using only the curve~\(F\). The total running time is approximatively~\(40\) minutes (see~\cite{programs}).

\begin{remark}\label{rk:EandF}
Alternatively, we can also combine information from both curves~$E$ and~$F$. From Theorem~\ref{T:overQ}, which uses Frey curve $E$, to prove Theorem~\ref{T:main} it suffices to deal with items~(\ref{item:partAiii}), (\ref{item:partBi}), and (\ref{item:partBii}) above. We note that the running time is essentially the same (and around~\(40\) minutes), but this proof requires less memory (1321.78MB vs. 947.59MB). In the next sections, we shall give proofs of Theorem~\ref{T:main} which avoid the most time-consuming step of the current approach, namely the computation of the space used in item~(\ref{item:partBii}) (and item~(\ref{item:partBiii})) below.

\end{remark}

\section{Two multi-Frey approaches to  Theorem~\ref{T:main} using Frey abelian varieties} 
\label{S:Frey7}

In this section, we will summarize the coming proofs of Theorem~\ref{T:main} and their contributions. We use the notation from Section~\ref{S:rks}.

The objective of our first `second-type' proof is to use higher dimensional Frey varieties introduced in~\cite{xhyper_vol1} `as much as possible'. More precisely, we shall use the properties of the abelian variety~$J/K$ constructed by Kraus (see~\S\ref{ss:recap} for a review of the definition and the useful results on~\(J\) from~\cite{xhyper_vol1}) to prove the following theorem.
\begin{theorem}\label{T:overQvariety}
 Let $p = 5$ or $p \geq 11$ be a prime. Then, 
the equation~\eqref{E:77p} with exponent~$p$
has no non-trivial primitive solutions $(a,b,c)$ 
satisfying $a \equiv 0 \pmod{2}$ and $b \equiv 1 \pmod{4}$.
\end{theorem}
The congruence hypothesis of this theorem implies $2 \nmid a+b$ which is the minimal assumption that makes the approach with~$J$ hopeful. Indeed,  when $2 \mid a+b$ we still face the fundamental problem of proving irreducibility of the residual representation~$\rhobar_{J,\Fp}$
(see \cite[Section~6]{xhyper_vol1}
and~\S\ref{ss:recap} for more details). Moreover, to make this proof efficient we need to add various techniques to the elimination step which allow for a large reduction of the required computational time; these techniques are detailed in~\cite[Section~9]{xhyper_vol1}
and can, in principle, also be applied to other instances of~\eqref{E:rrp} or to other Diophantine equations. See also~\S\ref{ss:recap} for a summary of these techniques in our current situation.

Let $(a,b,c)$ be a non-trivial primitive solution to~\eqref{E:77p}. By the symmetry of the equation, after switching $a$ and $b$, and negating both $a$ and $b$ if necessary, we may assume that the congruences 
$a \equiv 0 \pmod{2}$ and $b \equiv 1 \pmod{4}$
are satisfied if and only if $ab$ is even. 
Now the multi-Frey technique gives that $ab$ is even: this follows, using very quick computations, from items~\eqref{item:partBi} and~\eqref{item:partAi} in Section~\ref{S:Frey7first} which use the Frey curve $F$ and its quadratic twist by~\(-7\) respectively\footnote{Alternatively, one may also invoke Theorem~\ref{T:overQ} which regards the curve~$E$ combined with item~\eqref{item:partBi} in Section~\ref{S:Frey7first} which uses $F$. The running time would be essentially the same (and of less of a minute). See Table~\ref{table:second_proof} and~\cite{programs} for more details.}. Hence Theorem~\ref{T:overQvariety} implies Theorem~\ref{T:main}.


Our next -- `second-type' and last -- proof of Theorem~\ref{T:main} uses all three Frey varieties~\(E\), \(F\), \(J\) and is designed to minimize the total running time, which takes about $1$ minute (see~\cite{programs}). More precisely, we will use~$J$ to prove the following.
\begin{theorem}\label{T:overQvariety7adic}
 Let $p = 5$ or $p \geq 11$ be a prime. Then, 
the equation~\eqref{E:77p} with exponent~$p$
has no non-trivial primitive solutions $(a,b,c)$ 
satisfying
\[ a \equiv 0 \pmod{2}, \quad b \equiv 1 \pmod{4} \quad \text{ and } \quad 7 \mid a+b.\]
\end{theorem}
Assuming this result, we proceed similarly to the previous proof, but appealing to further information from the curve~$F$ and its twists. Namely, from Theorem~\ref{T:overQ} and items~\eqref{item:partAiii} and~(\ref{item:partBi}) of Section~\ref{S:Frey7first}, we have that~$7$ divides~$a + b$ and $ab$ is even. Now Theorem~\ref{T:overQvariety7adic} implies Theorem~\ref{T:main}.

For the convenience of the reader, for each of our different proofs of Theorem~\ref{T:main}, we indicate in the following tables the Frey varieties we used according to certain \(2\)- and \(7\)-adic conditions satisfied by~\(a\) and~\(b\). We also give the corresponding total running time. In each table, the use of any of the displayed Frey elliptic curves leads to essentially the same running time. We indicate in bold font which option was chosen in the code~\cite{programs}. As a matter of time-comparison example, we have also implemented the variant of our first proof mentioned in Remark~\ref{rk:EandF} (that is the proof using~\(E\) when available in Table~\ref{table:first_proof}). We note that the elliptic Frey curve~\(F\) can never be avoided: it is mandatory to deal with the case where~\(ab\) is odd and~\(7\) divides~\(a + b\).

\begin{table}[h!]
\centering
    \begin{minipage}[t]{0.46\linewidth}\centering
        \begin{tabular}{|c|c|c|}
\hline
 & $7\nmid a + b$  & $7\mid a + b$ \\
\hline
$2\nmid ab$ & \(E\) or $\boldsymbol{F^{(-7)}}$ & $F$ \\
\hline
$2\parallel ab$ & \(E\) or $\boldsymbol{F^{(-7\omega_2)}}$ & $F^{(\omega_2)}$ \\
\hline
$4\mid ab$ & $F^{(-7)}$ & \(E\) or $\boldsymbol{F}$ \\
\hline
\end{tabular}
\caption{`Frey elliptic curve only' proof(s) (\(\sim 40\) min.)}
\label{table:first_proof}
    \end{minipage}
    \hfill
    \begin{minipage}[t]{0.53\linewidth}\centering
       \begin{tabular}{|c|c|c|}
\hline
 & $7\nmid a + b$  & $7\mid a + b$ \\
\hline
$2\nmid ab$ & \(E\) or $\boldsymbol{F^{(-7)}}$ & $F$ \\
\hline
$2\parallel ab$ & \(J\) & $J$ \\
\hline
$4\mid ab$ & $J$ & \(J\) \\
\hline
\end{tabular}
\caption{Proof using \(J\) `as much as possible' (\(\sim 10\) min.)}
\label{table:second_proof}
    \end{minipage}

\end{table}

\begin{table}[h!]
\begin{tabular}{|c|c|c|}
\hline
 & $7\nmid a + b$  & $7\mid a + b$ \\
\hline
$2\nmid ab$ & \(\boldsymbol{E}\) or $F^{(-7)}$ & $F$ \\
\hline
$2\parallel ab$ & \(\boldsymbol{E}\) or $F^{(-7\omega_2)}$ & $J$ \\
\hline
$4\mid ab$ & $F^{(-7)}$ & \(J\) \\
\hline
\end{tabular}
\caption{Fastest proof of all (\(\sim 1\) min.)}
\label{table:third_proof}
\end{table}

We remark that both `second-type' proofs (from Tables~\ref{table:second_proof} and~\ref{table:third_proof}) involving the Frey abelian variety~\(J\) of dimension~\(3\) are faster than the `first-type' proof(s) using only Frey elliptic curves (either~\(F\) and its twists solely or a mix of~\(E\) and twists of~\(F\)). Moreover, as we shall see below, this improvement is in large part due to properties that are genuine of abelian varieties which are not elliptic curves. Also, the running time of a proof using the modular method is bounded from below by the time required to compute the relevant spaces of newforms. 
In this regard, it is usually an advantage to work with levels of smaller norms whenever possible; indeed, this is what makes our final proof the fastest of all and 
any proof using the space in item~\eqref{item:partBii} (or~\eqref{item:partBiii}) of Section~\ref{S:Frey7first} is slow. 
However, this is not always the case, as explained in~Remark~\ref{rk:Steinberg}.

When trying to solve~$\eqref{E:rrp}$ for~$r > 7$ or, more generally, carrying out Darmon's program for other signatures, the spaces of Hilbert newforms involved will quickly become very large. The following proofs will show how the additional structures of the Frey varieties of dimension~\(> 1\) can be exploited to reduce computations, despite the fact that we have to work with Jacobians of hyperelliptic curves.

\section{Proof of Theorems~\ref{T:overQvariety} and \ref{T:overQvariety7adic}}

\subsection{Summary of the results from~\cite{xhyper_vol1}}\label{ss:recap}

We summarize here the results from~\cite{xhyper_vol1} which we are using in the proofs of Theorems~\ref{T:overQvariety} and~\ref{T:overQvariety7adic}.

Let $(a,b,c)$ be a non-trivial primitive solution to equation~\eqref{E:77p} for exponent~$n = p$ prime, $p\ge 5$ and~$p\neq 7$ (see Section~\ref{S:rks}). Throughout this section, suppose~$a \equiv 0 \pmod{2}$ and~$b \equiv 1 \pmod{4}$.

We denote by $C = C_7(a,b)$ the hyperelliptic Frey curve constructed by Kraus in~\cite{kraushyper} with model
\begin{equation}\label{E:Kraus7}
y^2 = x^7 + 7 ab x^5 + 14 a^2 b^2 x^3 + 7 a^3 b^3 x + b^7 - a^7.
\end{equation}
Its discriminant is
\begin{equation*}
   \Delta(C)= -2^{12} \cdot 7^7 \cdot (a^7+b^7)^{6}.
\end{equation*}
Write~$J_7(a,b)$ for its Jacobian and simply denote by~\(J\) its base change to~\(K = \Q(\zeta_7)^+\). We note that when~\(ab = 0\), then from equation~\eqref{E:Kraus7}, we have an endomorphism $(x,y) \mapsto (\zeta_7 x, y)$ on~\(C_7(a,b)\) and hence~$J_7(a,b)$ has CM by~$\Q(\zeta_7)$. 

The following result is a special case of~\cite[Theorem~3.11]{xhyper_vol1}.

\begin{theorem}\label{T:GL2typeJ7} 
The abelian variety~\(J/K\) is of $\GL_2(K)$-type. More precisely,  there is an embedding
$$K \hookrightarrow \End_K(J) \otimes \Q$$
giving rise to a strictly compatible system of \(K\)-integral \(\lambda\)-adic representations
\begin{equation*}
\rho_{J,\lambda} : G_K\longrightarrow\GL_2(K_\lambda),
\end{equation*}
where~\(\lambda\) runs through the prime ideals in~\(K\).
\end{theorem}

Let~\(\Fq\) be a prime ideal in~\(K\) of good reduction for~\(J\) above a rational prime~\(q\). We denote by~\(\Frob_{\Fq}\) a Frobenius element at~\(\Fq\) in~\(G_K\) and write~\(a_\Fq(J) = \tr \rho_{J,\lambda}(\Frob_\Fq)\) for a prime ideal~\(\lambda \nmid q\). By strict compatibility, this definition is independent of the choice of such~\(\lambda\). Moreover, since~\(J\) is the base change of~\(J_7(a,b)\) which is defined over~\(\Q\), for any~\(\sigma \in G_\Q\), according to~\cite[Lemma~3.14]{xhyper_vol1} we have
\begin{equation} \label{E:traces}
\sigma(a_\Fq(J)) = a_{\sigma(\Fq)}(J).
\end{equation}

For a prime ideal~\(\lambda\) in~\(K\), the representation~$\rho_{J,\lambda}$ can be conjugated to take values in~$\GL_2(\calO_\lambda)$ where~$\calO_\lambda$ is the integer ring of~$K_\lambda$. By reduction modulo the maximal ideal, we get a representation
\[
\rhobar_{J,\lambda} : G_{K} \longrightarrow\GL_2(\F_\lambda),
\]
with values in the residue field~$\F_\lambda$ of~$K_\lambda$, and which is unique up to semisimplification and isomorphism.
Let~\(\lambda = \Fp\) be a prime ideal in~\(K\) above the rational prime~\(p\). We will need the following consequence of~\cite[Proposition~6.5]{xhyper_vol1}.

\begin{proposition}\label{P:case7}
The representation $\rhobar_{J,\Fp}$ is absolutely irreducible.
\end{proposition}
\begin{proof}
In the notation of~\cite[Proposition~6.5]{xhyper_vol1},
we have~\(g = 3\) and~\(m = 42\) as the prime~\(2\) is inert in~\(K = \Q(\zeta_7)^+ = \Q(\omega_1)\). Let $\infty_1$, $\infty_2$ and $\infty_3$ be the real places of $K$. The ray class group of modulus~\(\Fq_2\Fq_7\infty_1\infty_2\infty_3\) has order~\(2\). Recall from Section~\ref{S:rks} that the group of units in~\(K\) is generated by~\(\{-1,\epsilon_1,\epsilon_2\}\) with~\(\epsilon_1 = -\omega_1^2 + 1\) and~\(\epsilon_2 = \omega_1 + 1\). The prime numbers~\(p\ge5\), \(p\neq7\) dividing~\(\gcd\left( \Norm_{K/\Q}(\eps_1^{84}-1),\Norm_{K/\Q}(\eps_2^{84}-1)\right)\) are~\(13, 29, 43, 127, 337, 757\) and~\(2017\). They are all completely split in~\(K\). We check that for every prime ideal~\(\Fp_1\) over such prime~\(p\), the order of the ray class group of modulus~\(\Fq_2\Fq_7\Fp_1\infty_1\infty_2\infty_3\) is either~\(4\) or~\(12\). In particular, it is not divisible by~\(7\). It then follows from \emph{loc. cit.} that the representation~\(\rhobar_{J,\Fp}\) is absolutely irreducible. See~\cite{programs} for the details.
\end{proof}

\subsection{Proof of Theorem~\ref{T:overQvariety}}\label{ss:pf_2nd_type_1}

From Proposition~\ref{P:case7} and~\cite[Theorem~8.3]{xhyper_vol1} we have
\begin{equation}\label{E:iso7II}
\rhobar_{J,\Fp} \simeq \rhobar_{g,\mathfrak{P}}, 
\end{equation}
where~$g$ is in the set~\(S\) of Hilbert newforms over~$K$ of level $\Fq_2^2 \Fq_3 \Fq_7^2$, parallel weight $2$ and trivial cha\-racter, and~\(\mathfrak{P}\) is a prime ideal above~$p$ in~$K_g$, the  coefficient field of~$g$. Here~\(\rhobar_{g,\mathfrak{P}}\) denotes the reduction of the~\(\mathfrak{P}\)-adic representation associated with~\(g\). Moreover,  from conclusion~(iv) of the same theorem we have~$K \subset K_g$.

\subsubsection{General strategy}\label{ss:general_strategy}

The whole point is to contradict~\eqref{E:iso7II} for each form~\(g\in S\). From~\cite[Section~9]{xhyper_vol1}, recall that according to Shimura~\cite{shi78}, there is a left action \((\tau, g)\mapsto {}^\tau g\) of~\(G_\Q\) on~\(S\) such that
\begin{equation}\label{eq:Shimura}
a_\Fq\left(^\tau g\right) = \tau(a_\Fq(g)),
\end{equation}
for any prime ideal~\(\Fq\) in~\(\calO_K\). The Hecke constituent of~\(g\in S\) is defined as
\[
[g] = \left\{{}^\tau g : \tau\in G_\Q\right\}.
\]
For a form~\(g\in S\), we view its Hecke (or coefficient) field~\(K_g\) as a subfield of~\(\Qbar\) and note that two forms in the same Hecke constituent have conjugated Hecke fields, i.e. \(K_{{}^\tau g} = \tau(K_g)\).

The relevant space of cuspforms here is the new subspace generated by the set~\(S\). It has dimension~\(698\) and decomposes into~\(61\) Hecke constituents, which we can calculate with {\tt Magma} in about $8$ minutes. 

\begin{remark}\label{rk:Steinberg}
We note here that the computation of the space of Hilbert newforms over~$K$ of level~$\Fq_2^2 \Fq_7^2$, parallel weight $2$ and trivial character takes about $9$ hours, despite of its smaller level. Remarkably, we are able to avoid computing with it using refined level lowering as explained in~\cite[\S 8]{xhyper_vol1} (see in particular Remark~8.6 in {\it loc. cit.}). This is not the case for Theorem~\ref{T:other} though. 
\end{remark}

Assume that~\eqref{E:iso7II} holds for~\(g\in S\). Compared to the elliptic curve setting encountered before, new difficulties arise here in the elimination step which are due to the fact that~\(J\) has dimension~\(> 1\). In particular, the field of definition of the representation~\(\rhobar_{J,\Fp}\) need not be~\(\F_p\).

According to~\cite[Proposition~9.2]{xhyper_vol1}, there exists~\(\tau\in G_\Q\) satisfying~\(\tau(\mathfrak{P}) \cap \calO_K = \Fp\) so~\(\rhobar_{J,\Fp} \simeq \rhobar_{{}^\tau g,\tau(\mathfrak{P})}\) as representations with values in the residue field of~\(K_{{}^\tau g} = \tau(K_g)\) at the prime ideal~\(\tau(\mathfrak{P})\).

Consider a prime ideal~\(\Fq\) in~\(K\) above a rational prime~\(q\neq 2, 3, 7\). It follows from~\cite[\S5]{xhyper_vol1} that \(J\) has good reduction at~\(\Fq\) if \(q\nmid a^7 + b^7\) and bad multiplicative reduction otherwise. Therefore, for any~\(\sigma\in \Gal(K/\Q)\), the following congruence in the integer ring of~\(K_{{}^\tau g} = \tau(K_g)\) holds:
\[
a_{\sigma(\Fq)}({}^\tau g)) \equiv \left\{
\begin{array}{ll}
a_{\sigma(\Fq)}(J) \pmod{\tau(\mathfrak{P})} & \text{if \(q\nmid a^7 + b^7\)}; \\
\pm(N(\sigma(\Fq)) + 1) \pmod{\tau(\mathfrak{P})} & \text{if \(q\mid a^7 + b^7\)}.
\end{array}
\right.
\]

Assume first that~\(J\) has good reduction at~\(\Fq\). Write~\(u' = a_{\tau^{-1}(\Fq)}(J)\). Note that since~\(\Gal(K/\Q)\) is abelian we have~\(\sigma\tau^{-1}(\Fq) = \tau^{-1}\sigma(\Fq)\). Using~\eqref{E:traces} we have
\begin{align*}
     \Norm_{K_{^\tau g}/\Q}\left(a_{\sigma(\Fq)}(J) - a_{\sigma(\Fq)}(^\tau g)\right) & =  \Norm_{\tau(K_g)/\Q}\left(a_{\sigma(\Fq)}(J) - \tau(a_{\sigma(\Fq)}(g))\right) \\
      & = \Norm_{\tau(K_g)/\Q}\left(\tau\left(\tau^{-1}(a_{\sigma(\Fq)}(J)) - a_{\sigma(\Fq)}(g)\right)\right) \\
      &  = \Norm_{K_g/\Q}\left(a_{\tau^{-1}\sigma(\Fq)}(J)) - a_{\sigma(\Fq)}(g)\right) \\
      &  = \Norm_{K_g/\Q}\left(a_{\sigma\tau^{-1}(\Fq)}(J)) - a_{\sigma(\Fq)}(g)\right) \\
      &  = \Norm_{K_g/\Q}\left(\sigma(u') - a_{\sigma(\Fq)}(g)\right).
\end{align*}
Let~\(\mathfrak{S}\) be a subset of~\(\Gal(K/\Q)\). From our assumptions, we have
\begin{equation}\label{eq:elimination_good}
p \mid \gcd_{\sigma\in\mathfrak{S}} \Norm_{K_g/\Q}\left(\sigma(u') - a_{\sigma(\Fq)}(g)\right).
\end{equation}

Similarly, in the case of bad reduction we obtain
\begin{equation}\label{eq:elimination_bad}
p \mid M_{\Fq,\mathfrak{S}}(g) \coloneq \gcd_{\sigma\in\mathfrak{S}} \Norm_{K_g/\Q}\left(a_{\sigma(\Fq)}(g)^2 - (N(\Fq)+1)^2\right)
\end{equation}
as~\(N(\sigma(\Fq)) = N(\Fq) \in \Z\).

As explained in~\cite[Section~9]{xhyper_vol1}, there are now two main issues with this approach. Firstly, disproving the divisibility relations~\eqref{eq:elimination_good} and~\eqref{eq:elimination_bad} only discards~\(g\) but not the other forms in its Hecke constituent. Secondly \texttt{Magma} does not allow to compute~\(u' = a_{\tau^{-1}(\Fq)}(J)\) directly.

As a consequence, when~\(q\nmid a^7 + b^7\)  we introduce as in \emph{loc. cit.} the set~\(\mathcal{T}_{q}\) defined by
\begin{equation*}
    \mathcal{T}_{q} = \mathcal{T}_q(a,b) \coloneq  \left\{ a_{\Fq'}(J) : \Fq'\mid q\text{ in }\calO_K \right\}.
\end{equation*}
We point out here that the set~\(\mathcal{T}_q\) does depend on~\((a,b)\) (as~\(J\) denotes the base change to~\(K\) of~\(J_7(a,b)\)).
It follows from~\eqref{E:traces} that for any fixed prime ideal~\(\Fq'\) above~\(q\) in~\(K\) we have
\begin{equation}\label{eq:Tq_ppty}
\mathcal{T}_{q} = \left\{ a_{\sigma(\Fq')}(J) : \sigma \in \Gal(K/\Q) \right\} = \left\{ \sigma(a_{\Fq'}(J)) : \sigma \in \Gal(K/\Q) \right\}.
\end{equation}
In particular, \(\mathcal{T}_{q}\) consists of either exactly one element in $\Z$ or~\(3\) conjugated elements in $\calO_K$. The former case occurs when $q$ is inert in~$K$ for instance. Taking~\(\Fq' = \tau^{-1}(\Fq)\), we see that~\(u' = a_{\Fq'}(J) \in\mathcal{T}_q\) and from the previous discussion we have
\begin{equation}\label{eq:elimination_good2}
p \mid \prod_{u\in\mathcal{T}_q} \gcd_{\sigma\in\mathfrak{S}} \Norm_{K_g/\Q}\left(\sigma(u) - a_{\sigma(\Fq)}(g)\right).
\end{equation}
Let~$(x',y')\in\{0,\dots,q-1\}^2\backslash\{(0,0)\}$ be such that we have $(a,b) \equiv (x',y') \pmod{q}$. 
We have
\[
\mathcal{T}_q(a,b)  = \mathcal{T}_q(x',y').
\]
Putting everything together, we conclude that independently of the reduction type of~\(J\) at the prime ideals above~\(q\) (that is independently of~\(a^7 + b^7 \pmod{q}\)), we have
\begin{equation}\label{E:alltogether}
p \mid M_{\Fq, \mathfrak{S}}(g) \cdot \prod_{\substack{0 \leq x, y\leq q-1 \\  q\nmid x^7 + y^7}}\prod_{u\in\mathcal{T}_q(x,y)} \gcd_{\sigma\in\mathfrak{S}} \Norm_{K_g/\Q}\left(\sigma(u) - a_{\sigma(\Fq)}(g)\right).
\end{equation}
Now, the sets~\(\mathcal{T}_q\) are easily computed in \texttt{Magma} using the command \texttt{EulerFactor} (see~\cite[Section~9]{xhyper_vol1} and~\cite{programs} for the details). Moreover, as explained in~\cite[Section~9]{xhyper_vol1}, contradicting~\eqref{E:alltogether} discards the possibility that~\eqref{E:iso7II} holds for any of the forms in the Hecke constituent~\([g]\) of~\(g\). This solves the two issues mentioned above.

Moreover, the restriction~\(K \subset K_g\) reduces the set of newform representatives~\(g\) that possibly sa\-tisfy~\eqref{E:iso7II} from~61 down to~\(25\). The remaining 25 forms are indexed in~\cite{programs} by $i$ in the list 

\begin{center}
\begin{tabular}{@{}*{9}{@{\, }c}|*{2}{@{\, }c}|*{2}{@{\, }c}@{\, }*{5}{@{\! }|c@{\! }}}
\Circled{\Cc{12}} & \underline{16} & \Circled{\Cc{17}} & \Circled{18 19} & \Circled{20 21} & \Circled{22 23} & \Circled{\Cc{24}} & \underline{26} & \Circled{\Cc{28}} & \underline{33} & \underline{\underline{38}} & \underline{41} & \underline{42} & \underline{45} & \Circled{46 47} & \Circled{48 51} & \Circled{57 58} & \Circled{60 61}
\end{tabular}
\end{center}

\begin{tikzpicture}[overlay,remember picture,>=latex,shorten >=1pt,shorten <=1pt,very thin]
\draw[<->] (12) --++(0,-10pt) -| (28);
\draw[<->] (17) --++(0,10pt) -| (24);
\end{tikzpicture}

(forms linked by an arrow or within the same circle are twist of each other and will be dealt with simultaneously as explained below) and the degrees of their fields of coefficients are respectively given by
\begin{center}
\begin{tabular}{*{12}{c}|*{2}{c}|*{2}{c}|c|*{2}{c}|*{2}{c}|*{2}{c}|*{2}{c}}
3 & 3 & 3 & 3 & 3 & 3 & 3 & 3 & 3 & 3 & 3 & 3 & 9 & 9 & 12 & 12 & 15 & 18 & 18 & 21 & 21 & 36 & 36 & 54 & 54.
\end{tabular}
\end{center}

\begin{remark}
We want to stress the fact that this `instantaneous elimination', based on the constraint on the field of coefficients (given by~\(K \subset K_g\)), is unavailable with Frey elliptic curves.  See Remark~8.5 in~\cite{xhyper_vol1} for details.
\end{remark}

At this point, the `standard' application of the elimination procedure outlined above using, for each of the remaining forms, a few auxiliary primes to contradict~\eqref{E:iso7II} takes around 40 minutes, where forms 60 and 61 require about 15 minutes each, due to the large degree of their coefficient fields (see~\cite{programs}). To considerably reduce the computation time, we apply the various techniques from~\cite[Section~9]{xhyper_vol1} as explained below.

\subsubsection{Twisted forms}\label{sss:twisted_forms}

The first improvement we use is based on the fundamental observation that the remaining forms come in pairs of twisted forms (of possibly lower level) by (the multiplicative of) the mod~\(7\) cyclotomic character~\(\chi_7\) (which is quadratic when restricted to~\(G_K\)).


More precisely, let~\(f_q \in \{1,3\}\) be the common residue degree of the prime ideals in~\(K\) above a rational prime~\(q \neq 2,3, 7\). For~\(h\) a Hilbert newform of level dividing $\Fq_2^2 \Fq_3 \Fq_7^2$, parallel weight $2$ and trivial character, and for~\(\Fq\) a prime ideal above~\(q\) in~\(K\), we have
\begin{equation*}
 a_\Fq(h \otimes \chi_7|_{G_K}) = \chi_7(\Frob_\Fq) a_\Fq(h)
 = \left\{
\begin{array}{ll}
    a_\Fq(h) & \text{if }q^{f_q} \equiv 1\pmod{7}; \\
    -a_\Fq(h) & \text{if }q^{f_q} \equiv -1\pmod{7}. \\
\end{array} 
 \right.
\end{equation*}
Comparing fields of coefficients and Fourier coefficients at the unique prime ideal~\(\Fq_{11}\) above~\(11\) in~\(K\) (using the previous relation) when needed, we identify the 
following pairs of twisted newforms by~$\chi_7$  where at least one newform corresponds to $i$ in the list above:
\begin{enumerate}[(i)]
\item\label{item:forms1} the newform $i=38$ arises by twist of a from of level $\Fq_2^2 \Fq_3$;

\item\label{item:forms2} the newforms $i=16,26,33,41,42,45$ arise by twist of newforms at level $\Fq_2^2 \Fq_3 \Fq_7$;

\item\label{item:forms3} the pairs of newforms with field of coefficient equal to~$K$ and indexed by $(i,j) = (12,28)$, $(17,24)$, $(18,19)$, $(20,21)$ and~$(22,23)$ are quadratic twists by~$\chi_7$;

\item\label{item:forms4} the pairs of newforms indexed by $(i,j) = (46,47)$, $(48,51)$, $(57,58)$ and~$(60,61)$ are quadratic twists by~$\chi_7$; this follows from the condition~\(K\subset K_g\) and the fact that these pairs are determined by the degree of their fields of coefficients, which are 18, 21, 36 and~54, respectively (all newforms of ``smaller'' level have a field of coefficients whose degree is~\(\le 15\)). 
\end{enumerate}
In the previous list, we have double-underlined the unique form in case~\eqref{item:forms1}, underlined the forms in case~\eqref{item:forms2}, and circled the pairs of forms in cases \eqref{item:forms3}-\eqref{item:forms4} (or linked the two forms in a pair by an arrow when they were not next to each other) respectively.

This calculation, including the computation of  newforms over $K$ of parallel weight~2, trivial character and levels~$\Fq_2^2 \Fq_3$, $\Fq_2^2 \Fq_3 \Fq_7$ or $\Fq_2^2 \Fq_3 \Fq_7^2$, takes a few seconds (see~\cite{programs}).

\subsubsection{Symmetry}\label{sss:symmetry}

Our second improvement relies on the fact that~\(C_7(b,a)\) is the quadratic twist of~\(C_7(a,b)\) by~\(-1\) (see~\cite[Proposition~3.3]{xhyper_vol1}). Write~\(J'\) for~\(J_7(b,a)\) base changed to~\(K\). For~\(\Fq\) a prime ideal above~\(q\) in~\(K\), we have (see~\cite[\S9.8]{xhyper_vol1}):
\begin{equation*}
 a_\Fq(J) 
 = \left(\frac{-1}{q}\right)^{f_q}a_\Fq(J')
 = \left\{
\begin{array}{ll}
a_\Fq(J') & \text{if }q^{f_q} \equiv 1\pmod{4}; \\
-a_\Fq(J') & \text{if }q^{f_q} \equiv -1\pmod{4}. \\
\end{array}
\right.
\end{equation*}

\subsubsection{Eliminating the forms}

Let us now consider a Hilbert newform~\(h\) as above (i.e. \(h\) is of level dividing $\Fq_2^2 \Fq_3 \Fq_7^2$, parallel weight $2$ and trivial character) and any subset~\(\mathfrak{S}\) of~\(\Gal(K/\Q)\).

If~\(q^{f_q}\equiv 1\pmod{4}\), we set
\begin{equation}\label{eq:Bq1}
B_{\Fq,\mathfrak{S}}(h) \coloneq  M_{\Fq,\mathfrak{S}}(h)  \cdot
\prod_{\substack{0\leq x \leq y\leq q-1 \\ q\nmid x^7 + y^7}}  \prod_{u\in\mathcal{T}_q(x,y)}\gcd_{\sigma\in\mathfrak{S}}\Norm_{K_h/\Q}\left(\sigma(u) - a_{\sigma(\Fq)}(h)\right)
\end{equation}
where the notation is exactly as above; we highlight the difference that the outermost product differs from~\eqref{E:alltogether} in that it contains the additional restriction $x \leq y$. Due to the large degrees of the coefficient fields and the size of the auxiliary primes involved in the elimination procedure, this improvement turns out to be crucial.

Else, if~\(q^{f_q}\equiv -1\pmod{4}\), we set
\begin{equation}\label{eq:Bq2}
B_{\Fq,\mathfrak{S}}(h) \coloneq  M_{\Fq,\mathfrak{S}}(h)  \cdot
\prod_{\substack{0\leq x \leq y\leq q-1 \\ q\nmid x^7 + y^7}}  \prod_{u\in\mathcal{T}_q(x,y)}\gcd_{\sigma\in\mathfrak{S}}\Norm_{K_h/\Q}\left(\sigma(u)^2 - a_{\sigma(\Fq)}(h)^2\right).
\end{equation}
From the discussion above, still assuming that~\eqref{E:iso7II} holds for~\(g\in S\), we have that $p \mid B_{\Fq,\mathfrak{S}}(g)$ in all cases; furthermore, when~$\mathfrak{S} = \Gal(K/\Q)$, we obtain a quantity~$B_q(g) = B_{\Fq,\mathfrak{S}}(g)$ depending only on~$q$ and~$[g]$.

Assume now that~\(h\) satisfies~\(g = h \otimes \chi_7\). Then we have that~\(M_{\Fq,\mathfrak{S}}(g) = M_{\Fq,\mathfrak{S}}(h)\). According to the discussion above, if~\(q^{f_q}\equiv 1\pmod{7}\) or~\(q^{f_q}\equiv -1\pmod{4}\), then we have~\(B_{\Fq,\mathfrak{S}}(g) = B_{\Fq,\mathfrak{S}}(h)\). 

Else, if~\(q^{f_q}\equiv -1\pmod{7}\) and~\(q^{f_q}\equiv 1\pmod{4}\), then~$B_{\Fq,\mathfrak{S}}(g)B_{\Fq,\mathfrak{S}}(h)$ divides  
\begin{equation}\label{eq:last_case}
M_{\Fq,\mathfrak{S}}(g)^2 \cdot
\prod_{\substack{0\leq x \leq y\leq q-1 \\ q\nmid x^7 + y^7}}  \prod_{u\in\mathcal{T}_q(x,y)}\gcd_{\sigma\in\mathfrak{S}}\Norm_{K_g/\Q}\left(\sigma(u)^2 - a_{\sigma(\Fq)}(g)^2\right),  
\end{equation}
and proving that~$p$ does not divide this number implies 
$p \nmid B_{\Fq,\mathfrak{S}}(g)$ and $p \nmid B_{\Fq,\mathfrak{S}}(h)$.

In both cases, discarding one form allows to discard the other one as well.

We now look at the forms~\(i\) for~\(i\) as in~\eqref{item:forms1}-\eqref{item:forms4} in turn using prime ideals above specific rational primes~\(q\neq 2,3,7\) with the various techniques we have introduced before. 

Suppose $i=38$. The quadratic twist of~$\rhobar_{g,\mathfrak{P}}$ by~$\chi_7|_{G_K}$ is unramified at~$\Fq_7$ by item~(\ref{item:forms1}). Then the same holds for~$\rhobar_{J,\Fp}$ by~\eqref{E:iso7II}. This is incompatible with the type bad reduction reduction of~$J$ at~\(\Fq_7\) described in \cite[Propositions~5.7 and~5.9]{xhyper_vol1}.

Each of the forms $i =16,26,33,41,42$ or $45$ in item~(\ref{item:forms2}) can be eliminated for all~$p$ by compu\-ting~$B_{\Fq,\mathfrak{S}}(g)$ for a combination of the auxiliary primes above $q \in \{5, 11, 13\}$ and various sets~\(\mathfrak{S}\). As a matter of example, we describe more precisely the procedure used for the form~\(f\) corresponding to~\(i = 16\). Using~\eqref{eq:Bq1} (with~\(h\) replaced by~\(f\)), we first compute~\(B_{\Fq_{13},\{1\}}(f)\) for a prime ideal~\(\Fq_{13}\) in~\(K\) above~\(13\). The prime divisors~\(p\neq 2, 3, 7\) of this number are~\(13, 29, 43, 281\). These are now the remaining primes we have to deal with. We then compute~\(B_{\Fq_{11},\{1\}}(f)\) using~\eqref{eq:Bq2} (here, according to~\eqref{eq:elimination_bad} and the discussion after~\eqref{eq:Tq_ppty}, the set~\(\mathfrak{S}\) does not matter as~\(11\) is inert in~\(K\)) and we find that the only prime not discarded is~\(p = 13\). We finally use refined elimination as described in~\cite[\S9.11]{xhyper_vol1} using the three prime ideals above~\(29\) in~\(K\) to deal with this last exponent. Similar arguments (but simpler since no refined elimination is required) apply to eliminate the other forms~\(i = 26, 33, 41, 42\) or~\(45\) (see \cite{programs}).

We now deal with the forms in items~(\ref{item:forms3}) and~(\ref{item:forms4}), i.e., the newforms that are quadratic twists by $\chi_7$ of other forms in the same level~$\Fq_2^2 \Fq_3 \Fq_7^2$. Each of these nine pairs, except the last pair~\((60,61)\), can be eliminated for all~$p$ using a combination of auxiliary primes  above $q \in \{5, 11, 13,17,29\}$ and the ideas explained above. As a matter of example again, we provide details on the precise~\((\Fq,\mathfrak{S})\)'s used to eliminate the pair~\((g,h) = (g, g\otimes\chi_7)\) corresponding to~\((i,j) = (22,23)\). Since~\(13\) splits in~\(K\) and satisfies~\(13 \equiv -1 \pmod{7}\) and~\(13 \equiv 1 \pmod{4}\), we compute the quantity in~\eqref{eq:last_case} with~\(\Fq = \Fq_{13}\) being any prime ideal above~\(13\) in~\(K\) and~\(\mathfrak{S} = \{1\}\). We find that the only primes dividing this number are~\(13, 29, 41, 43, 71, 113, 139, 181, 239, 307\), and~\(379\). The prime~\(11\) is inert in~\(K\) and we have~\(11^3 \equiv 1\pmod{7}\) (and also~\(11^3 \equiv -1\pmod{4}\)). Therefore, computing~\(B_{\Fq_{11},\mathfrak{S}}(g) = B_{\Fq_{11},\mathfrak{S}}(h)\) with~\(\Fq_{11}\) the only prime above~\(11\) in~\(K\) (here again the set~\(\mathfrak{S}\) does not matter as~\(11\) is inert), we deduce that~\(p\in\{13,29\}\). The prime~\(17\) is inert in~\(K\) and satisfies~\(17^3 \equiv -1 \pmod{7}\) and~\(17^3 \equiv 1 \pmod{4}\). Computing the quantity in~\eqref{eq:last_case} with~\(\Fq = \Fq_{17}\) being the only prime ideal above~\(17\) in~\(K\) (and arbitrary set~\(\mathfrak{S}\)), we discard~\(p = 29\). Therefore, it remains to deal with~\(p = 13\). We do it using refined elimination with two auxiliary prime ideals above the (split) prime~\(29\) as explained in~\cite[\S9.11]{xhyper_vol1}. All of the other pairs in items~(\ref{item:forms3}) and~(\ref{item:forms4}), except the one corresponding to~\((i, j) = (60, 61)\) can be dealt with using similar (yet simpler as no refined elimination is required) arguments.

Finally, due to the large degree of their coefficient field, forms corresponding to $(i,j) = (60,61)$ require special care, as we now explain, following, in our situation, the discussion in~\cite[\S9.6]{xhyper_vol1}.

Recall that~\(S\) denotes the set of Hilbert newforms over~$K$ of level $\Fq_2^2 \Fq_3 \Fq_7^2$, parallel weight $2$ and trivial cha\-racter. Then, there is a right action~\((f,\sigma)\mapsto f^\sigma\) of~\(\Gal(K/\Q)\) on~\(S\) such that
\[
a_\Fq\left(f^\sigma\right) = a_{\sigma(\Fq)}(f),
\]
for any prime ideal~\(\Fq\) in~\(\calO_K\) (see \cite[\S 6.3]{Gelbart-Boston} for a high-level description).

This action commutes with that defined in~\eqref{eq:Shimura} and thus~\(\Gal(K/\Q)\) acts on the set of Hecke constituents~\(\{[f] : f\in S\}\) as~\(([f],\sigma)\mapsto [f^\sigma]\). We have $[K : \Q] =3$ so the orbits of the action of $\Gal(K/\Q)$ are of size 1 or 3. Since forms~$60, 61$ are the only forms whose field of coefficients has degree 54, their Hecke constituent must be fixed by the action of $\Gal(K/\Q)$.

Let $g$ be the newform corresponding to $i=60$. A quick check of the Fourier coefficients at the primes above~$13$ shows that this form is not a base change from~\(\Q\), by which we mean that~\(g^{\sigma} \neq g\), for any non-trivial~\(\sigma\in \Gal(K/\Q)\).

Recall that~\(\sigma_0\) denotes a fixed generator of~\(\Gal(K/\Q)\). Then, there exists~\(\tau_0\in G_\Q\) such that we have~\(g^{\sigma_0} = {}^{\tau_0} g\). Moreover~\(\tau_0\) acts through its restriction to~\(K_g\) and this restriction is an automorphism of~\(K_g\) since~\(\tau_0(a_\Fq(g)) = a_{\sigma_0(\Fq)}(g)\in K_g\) for all prime ideals~\(\Fq\) in~\(\calO_K\). 

Let~\(E_g\) be the subfield of~\(K_g\) fixed by the subgroup of~\(\Aut(K_g)\) generated by~\(\tau_0\). Let us show that~\(\tau_0\) has order~\([K:\Q]\) in~\(\Aut(K_g)\). For every~\(k\geq1\), we have~\(g^{\sigma_0^k} = {}^{\tau_0^k}g\) as the two actions commute. If~\(\sigma_0^k\) is trivial, then we have~\(\tau_0^k(a_\Fq(g)) = a_\Fq({}^{\tau_0^k} g) = a_\Fq(g)\) for all prime ideals~\(\Fq\) in~\(\calO_K\), and hence~\(\tau_0^k\) is also trivial. Conversely, if~\(\tau_0^k\) is trivial, then we have~\(g^{\sigma_0^k} = g\) and we conclude that~\(\sigma_0^k\) is trivial by the fact that~\(g\) is not a base change from~\(\Q\).

Therefore, we have~\([K_g:\Q] = [K_g:E_g] [E_g:\Q] = [K:\Q] [E_g:\Q]\), and hence
\[
[E_g:\Q] = [K_g:\Q]/[K:\Q] = [K_g:K] = 18.
\]
Assume now that~\(q\) is inert in~\(K\) and write~\(q\calO_K = \Fq\). 

Suppose further that~\(J\) has good reduction at~\(\Fq\). On one hand, we know that~\(\mathcal{T}_q = \{a_\Fq(J)\}\) consists of a single (rational) integer. On the other hand, we have 
\[
\tau_0(a_\Fq(g)) = a_\Fq({}^{\tau_0} g) = a_\Fq(g^{\sigma_0}) = a_{\sigma_0(\Fq)}(g) = a_\Fq(g)
\]
and hence~\(a_\Fq(g)\in E_g\). Therefore, we have
\[
\Norm_{K_g/\Q}\left(a_\Fq(J) - a_\Fq(g)\right) = \Norm_{E_g/\Q}\left(a_\Fq(J) - a_\Fq(g)\right)^{[K_g:E_g]}.
\]
Similarly, if~\(J\) has bad multiplicative reduction at~\(\Fq\), then we have
\begin{equation*}
\Norm_{K_g/\Q}\left(a_{\Fq}(g)^2 - (N(\Fq)+1)^2\right) = \Norm_{E_g/\Q}\left(a_{\Fq}(g)^2 - (N(\Fq)+1)^2\right)^{[K_g:E_g]}
\end{equation*}
as~\(N(\Fq) \in \Z\).

We conclude that independently of the reduction type of~\(J\) at the prime ideals above the inert prime~\(q\) (that is independently of~\(a^7 + b^7 \pmod{q}\)), if~\eqref{E:iso7II} holds for~\(g\), then we have
\begin{equation}\label{eq:divisibility_with_E}
p \mid \Norm_{E_g/\Q}\left(a_{\Fq}(g)^2 - (N(\Fq)+1)^2\right) \cdot \prod_{\substack{0 \leq x, y\leq q-1 \\  q\nmid x^7 + y^7}} \Norm_{E_g/\Q}\left(a_\Fq(J) - a_{\Fq}(g)\right).
\end{equation}
We note that the right-hand side of~\eqref{eq:divisibility_with_E} does not change if~\(g\) is replaced by a conjugate~\({}^\tau g\). Indeed, this follows from the above identities and the fact that 
$[K_{\tau^g}:E_{\tau^g}] = [K_g:E_g]$, where this equality of degrees holds 
because~\(\left({}^\tau g\right)^{\sigma_0} = {}^{\tau\tau_0\tau^{-1}}\left({}^\tau g\right)\) and hence~\(E_{{}^\tau g} \coloneq  K_{{}^\tau g}^{\langle\tau\tau_0\tau^{-1}\rangle} = \tau(E_g)\).

Reasonning as in~\S\S\ref{sss:twisted_forms}-\ref{sss:symmetry}, we slightly modify the right-hand side of~\eqref{eq:divisibility_with_E} in order to avoid useless computations (by using the symmetry of~\(C_7(a,b)\)) and to discard both~\(g\) and~\(g\otimes\chi_7\) (corresponding to the form number~\(61\)) at once.

We apply this method with the inert primes~\(q = 5, 11\). This allows for some time improvement, but unfortunately, those primes are still insufficient to eliminate the pair~\((g, g\otimes\chi_7)\) for the exponents~\(p\in\{5,11,23,47\}\).

Therefore, we use a prime above~$q=13$ which (totally) splits in~$K$. Given the large amount of repetitive calculations involved in the elimination step, having nice polynomials describing the field of coefficients and its subfields can also accelerate the elimination of this newform. With {\tt Magma} we can efficiently write 
$K_g/\Q$ as a relative extension $K_g/K_6/\Q$ 
where $K_6$ is of degree~$6$ such that $K \subset K_6$ and the polynomials defining $K_g/K_6$ and $K_6/\Q$ have nicer coefficients. This cuts down the time used when dealing with auxiliary primes above~13 considerably (see~\cite{programs}).

Using all these improvements, the total running time used to eliminate the pair of forms corresponding to~\((i,j) = (60,61)\) finally comes down to less that one minute, compared to~\(15\) minutes for each of these forms without these techniques (see the discussion at the end of~\S\ref{ss:general_strategy}).

This achieves the proof of Theorem~\ref{T:overQvariety} and hence of Theorem~\ref{T:main} according to the argument given in Section~\ref{S:Frey7}. In total, this proof takes approximately \(10\) minutes.

\begin{remark}
We notice that these ideas can actually be applied to all 25 forms we need to eliminate.
Indeed, from the degrees listed above, we see there are at most 2 Hecke constituents of newforms whose field of coefficients has degree~$d_g > 3$, and hence the
same arguments used for the forms corresponding to~$(i,j) = (60,61)$ can be used. In fact, with a few comparisons of Fourier coefficients, we can also deal with the forms with $d_g=3$. However, the time gains are negligible compared to the forms~$(i,j) = (60,61)$ which took~\(3/4\) of the overall computational time of the `standard' elimination step.
\end{remark}

\subsection{Proof of Theorem~\ref{T:overQvariety7adic}}

Suppose further that~$7 \mid a+b$.

Let $\Fp \mid p$ in~$K$.
From Proposition~\ref{P:case7} and~\cite[Corollary~8.7]{xhyper_vol1} we have
\begin{equation*}
\rhobar_{J,\Fp} \simeq \rhobar_{g,\mathfrak{P}} \otimes \chi_7|_{G_K}, 
\end{equation*}
where~$g$ is a Hilbert newform over~$K$ of level $\Fq_2^2 \Fq_3 \Fq_7$, parallel weight $2$, trivial character and a prime ideal~$\mathfrak{P} \mid p$ in~$K_g$, the  coefficient field of~$g$. Furthermore, we have~$K \subset K_g$. Computing the corresponding space of newform and checking whether $K \subset K_g$ leaves us with 6 newforms to eliminate. Now, similar (but simpler) arguments as those used in the proof of Theorem~\ref{T:overQvariety}, using the auxiliary primes $q=5,11,13$ eliminate all 6 forms for all exponents~$p$. The total running time here is less than one minute (see Section~\ref{S:Frey7} and~\cite{programs}).

Together with the argument given in Section~\ref{S:Frey7}, this concludes our last and fastest-of-all proof of Theorem~\ref{T:main}.

\section{Proof of Theorem~\ref{T:other} and Corollary~\ref{C:reduction2CM}}

\subsection{Proof of Theorem~\ref{T:other}}

We first prove Part~(\ref{T:other_item1}). Assume~$n = p \ge 5$ is a prime satisfying~\(p \neq 7\) (see Section~\ref{S:rks}). Since the levels are not too big, we cover this case in the same way as for Theorem~\ref{T:main} using the proof that relies only on the Frey curve $F$ (see Section~\ref{S:Frey7first}), except that we only need the case $ab$ odd from Part~\ref{item:partA}, and all of Part~\ref{item:partB}, with the change that the prime $\Fq_3$ is removed from the levels of Hilbert newforms computed (see~\cite{programs}).

Let us now prove Part~(\ref{T:other_item2}). Let~$(a,b,c)$ be a non-trivial primitive solution to the equation~\eqref{main-equ} with $n=p$ prime, $p\notin \left\{ 2, 3, 7 \right\}$. By Part~(\ref{T:other_item1}), we have that $2 \mid ab$ and $7 \nmid a + b$. So up to switching the roles of $a$ and $b$ and/or negating both $a$ and $b$, we may assume that $a \equiv 0 \pmod 2$ and $b \equiv 1 \pmod 4$.

Let $\Fp \mid p$ in~$K$. From Proposition~\ref{P:case7} and~\cite[Theorem~8.3]{xhyper_vol1} we have
\begin{equation}
\rhobar_{J,\Fp} \simeq \rhobar_{g,\mathfrak{P}}, 
\end{equation}
where~$g$ is in the new subspace $S$ of Hilbert newforms of level $\Fq_2^2 \Fq_{7}^2$, parallel weight $2$, and trivial character. Moreover,  from conclusion~(iv) of that theorem we have~$K \subset K_g$ and $K_g$ is the coefficient field of $g$. 

The space $S$ has dimension $27$ and decomposes into $9$ Hecke constituents. Its initialization in~\texttt{Magma} takes about~\(9\) hours (see Remark~\ref{rk:Steinberg}). There are $4$ Hilbert newforms $g$ indexed by $4,5,6,8$ which have field of coefficients $K_g \supseteq K$ (and in fact $K_g = K$). Forms $4, 6, 8$ can be eliminated using~\(\Fq_3\) and any prime above~\(13\) as auxiliary primes. Form~$5$ cannot be eliminated and corresponds to~$J(0,1)$ (see~\cite{programs}).

Finally, because switching the roles of $a$ and $b$ and/or negating both $a$ and $b$ changes $\rhobar_{J(0,1),\Fp}$ by a character of order dividing $2$, we obtain the conclusion in the form stated.

\subsection{Proof of Corollary~\ref{C:reduction2CM}}

The variety $J(a,b)/K$ satisfies the hypotheses in \cite[Conjecture 4.1]{DarmonDuke}. Thus for large enough~$p$, the representation~$\rhobar_{J(a,b),\Fp}$ has image not contained in the normalizer of a Cartan subgroup. This contradicts the isomorphism $\rhobar_{J(a,b),\Fp} \simeq \rhobar_{J(0,1),\Fp} \otimes \chi$ in Theorem~\ref{T:other} because $\rhobar_{J(0,1),\Fp}$ is contained in the normalizer of a Cartan since $J(0,1)$ has CM (see~\S\ref{ss:recap}).


\end{document}